\documentclass[a4paper]{article}

\RequirePackage{amsthm,amsmath,amsfonts,amssymb}
\usepackage{bm}
\usepackage{bbm}
\usepackage{mathtools}
\usepackage{enumitem}
\usepackage{xcolor}
\usepackage{graphicx}
\usepackage{float}
\usepackage{natbib}
\usepackage{subcaption}
\usepackage{tcolorbox}
\usepackage{hyperref}
\usepackage{geometry}
\usepackage{setspace}

\DeclarePairedDelimiter\abs{\lvert}{\rvert}
\DeclarePairedDelimiter\norm{\lVert}{\rVert}
\DeclarePairedDelimiter\normw{\lVert}{\rVert_{\textup{b}}}
\DeclarePairedDelimiter\normww{\lVert}{\rVert_{\textup{b}}}
\DeclarePairedDelimiter\normwww{\lVert}{\rVert_{\sigma}}
\usepackage{bibunits}
\defaultbibliographystyle{apalike}

\newcommand{\diff}{\textup{d}}
\theoremstyle{plain}

\newtheorem{theorem}{Theorem}
\newtheorem{example}{Example}
\newtheorem{lemma}[theorem]{Lemma}
\newtheorem{corollary}[theorem]{Corollary}
\newtheorem{proposition}[theorem]{Proposition}

\newtcolorbox{commentbox}{
    colframe=black,      
    colback=white,       
    width=\linewidth,    
    boxrule=1pt,         
    arc=0pt,             
    outer arc=0pt,       
    left=8pt,            
    right=8pt,           
    top=8pt,             
    bottom=8pt,          
}

\theoremstyle{definition}

\newtheorem{remark}{Remark}
\newtheorem{assumption}{Assumption}

\newcommand{\expv}{\mathbb{E}}
\newcommand{\prob}{\mathbb{P}}
\newcommand{\cov}{\mathbb{C}\textup{ov}}
\newcommand{\var}{\mathbb{V}\textup{ar}}
\newcommand{\ind}{\mathbbm{1}}

\DeclareMathOperator*{\argmax}{arg\,max}
\DeclareMathOperator*{\argmin}{arg\,min}

\title{Inference for sparsely sampled Gauss--Markov processes under  outcome-dependent dropout}
\author{Alexander Aue, Siegfried H\"o\-rmann, Maximilian Ofner}

\begin{document}


\maketitle

\begin{abstract}
    We consider sparsely observed functional data generated by a latent Gauss--Markov process modeled through a linear stochastic differential equation (SDE). Motivated by an application to tumor growth data, we allow for outcome-dependent dropout, where sampling terminates once the most recent observation exceeds a prescribed threshold. Such observation schemes arise naturally in longitudinal studies and violate the fundamental missing completely at random assumption commonly imposed in the analysis of partially observed functional data. We develop a likelihood-based estimation framework that,
    in contrast to existing moment-based methods, avoids the bias induced by outcome-dependent dropout. We establish convergence rates for the proposed estimators and show that they are minimax optimal up to logarithmic factors, with the diffusion coefficient admitting a faster rate than the drift coefficients under a random initial condition for the SDE. An application to tumor growth data further demonstrates the advantages of a fully probabilistic functional data model for tasks beyond second-order inference, including prediction bands, first-passage times, and tumor-age estimation.
\end{abstract}

\begin{bibunit}

\section{Introduction}
\label{sec:intro}

Functional data analysis (FDA) has emerged as a central area of modern statistics. In FDA, each observational unit is viewed as a realization of a continuous-time stochastic process, $(X_i(t)\colon t \in [0,T])$ for $i = 1,\ldots,N$. 
Functional data arise in a broad range of scientific applications, including medicine, engineering, economics, and the environmental sciences. Comprehensive treatments of FDA can be found in~\cite{ramsay1997functional},  \cite{ferraty2006nonparametric},
\cite{horvath2012inference}, and \cite{hsing2015theoretical}. 

The mean function $\mu(t)=\expv[X(t)]$ and the covariance function $\cov(X(s), X(t))$ play a fundamental role in FDA. They underpin many of its core methodologies, including functional principal component analysis, dimension reduction, and functional regression and prediction. Consequently, a substantial body of research has focused on estimating these functions under a wide variety of sampling schemes, including sparse, dense, and partially observed designs. Nevertheless, mean and covariance generally do not fully characterize the distribution of the underlying stochastic process.

Suppose that the stochastic processes $X_i$, for $i=1,\ldots,N$, satisfy the stochastic differential equation (SDE)
\begin{align}\label{eq:GeneralSDE}
\diff X_i(t) = b(t, X_i(t)) \,\diff t + \sigma(t, X_i(t)) \,\diff B_i(t), 
\qquad t \in [0,T], 
\qquad X_i(0) \sim X_0,
\end{align}
where $b\colon [0,T] \times \mathbb{R} \to \mathbb{R}$ denotes the \emph{drift}, $\sigma\colon [0,T] \times \mathbb{R} \to \mathbb{R}$ the \emph{diffusion function}, $B_i = (B_i(t)\colon t \in [0,T])$ a Brownian motion, and $X_0$ the initial distribution. Under conditions ensuring existence and uniqueness of the solution, the SDE completely specifies the distribution of the process and thus provides a probabilistic model for functional data. This enables inference on distributional features that cannot be recovered from the mean and covariance functions alone. 

One important application is the construction of prediction bands. Suppose that the trajectory of $X_i$ has been observed on $[0,t_0]$, and that the objective is to construct a band covering its future evolution on $(t_0,T]$ with a prescribed probability. Such prediction bands require inference on the conditional distribution of the future trajectory given its observed history. Another important application concerns first-passage times, which play a central role in the analysis of biomarker trajectories. If $X_i$ represents the evolution of a patient's biomarker, a clinically relevant event may occur when the biomarker first exceeds a threshold. Examples include the time until a tumor reaches a critical size or a patient's viral load exceeds a treatment-failure threshold. In general, these quantities cannot be inferred from the mean and covariance functions alone.

Although distributional models based on stochastic differential equations are widely used in financial mathematics, they have only recently begun to attract attention in functional data analysis. Recent work by \citet{comte2020nonparametric}, \citet{denis2021ridge}, \citet{mohammadi2024nonparametric}, and \citet{kokoszka2025functional} models functional data through SDEs and develops estimation procedures for a variety of observation schemes, ranging from fully observed trajectories to sparse designs. Closely related is the work of \citet{zhou2024dynamic}, who also consider an SDE model for sparsely observed functional data but focus on trajectory reconstruction rather than parameter estimation.

The primary objective of this paper is to advance distributional modeling in FDA by proposing a new estimation methodology. To this end, we consider a parsimonious class of SDEs that gives rise to Gauss--Markov processes. We demonstrate the practical utility of the proposed model through an application to tumor growth data, where we address several clinically relevant problems, including the construction of prediction bands and the estimation of tumor age.

The principal methodological contribution of the paper lies in the proposed estimation procedure. Existing approaches are predominantly based on method-of-moments techniques. In particular, \citet{kokoszka2025functional} for densely observed functional data and \citet{mohammadi2024nonparametric} for sparsely observed data first estimate the mean and covariance functions and then exploit the relationship between these quantities and the underlying model parameters. In contrast, our approach estimates the model parameters directly from the observed trajectories, thereby avoiding a moment-estimation step.

Beyond its statistical efficiency, the proposed likelihood framework naturally accommodates sampling schemes subject to \emph{outcome-dependent dropout} (ODD). We adopt the term \emph{dropout} from the longitudinal data literature, where it refers to subjects becoming unavailable during the course of a study (see, for instance, 
\cite{little1995modeling},
\cite{lindsey2000}, \cite{hoganetal2004}, \cite{wilsonetal2025}). 
In the context of functional data, ODD describes sampling schemes in which the decision to continue or terminate observations depends on the most recently observed value of the process. Such designs arise in diverse applications. Examples include battery degradation studies, where capacity-loss curves are observed until the capacity falls below a prescribed threshold and the battery is declared to have reached its end of life \citep{ahwiadi2025battery} or clinical trials, where treatment is discontinued upon exceeding a pre-defined biomarker safety threshold \citep{Watkinsetal1994}.

\begin{figure}
    \centering
    \includegraphics[width=0.8\linewidth]{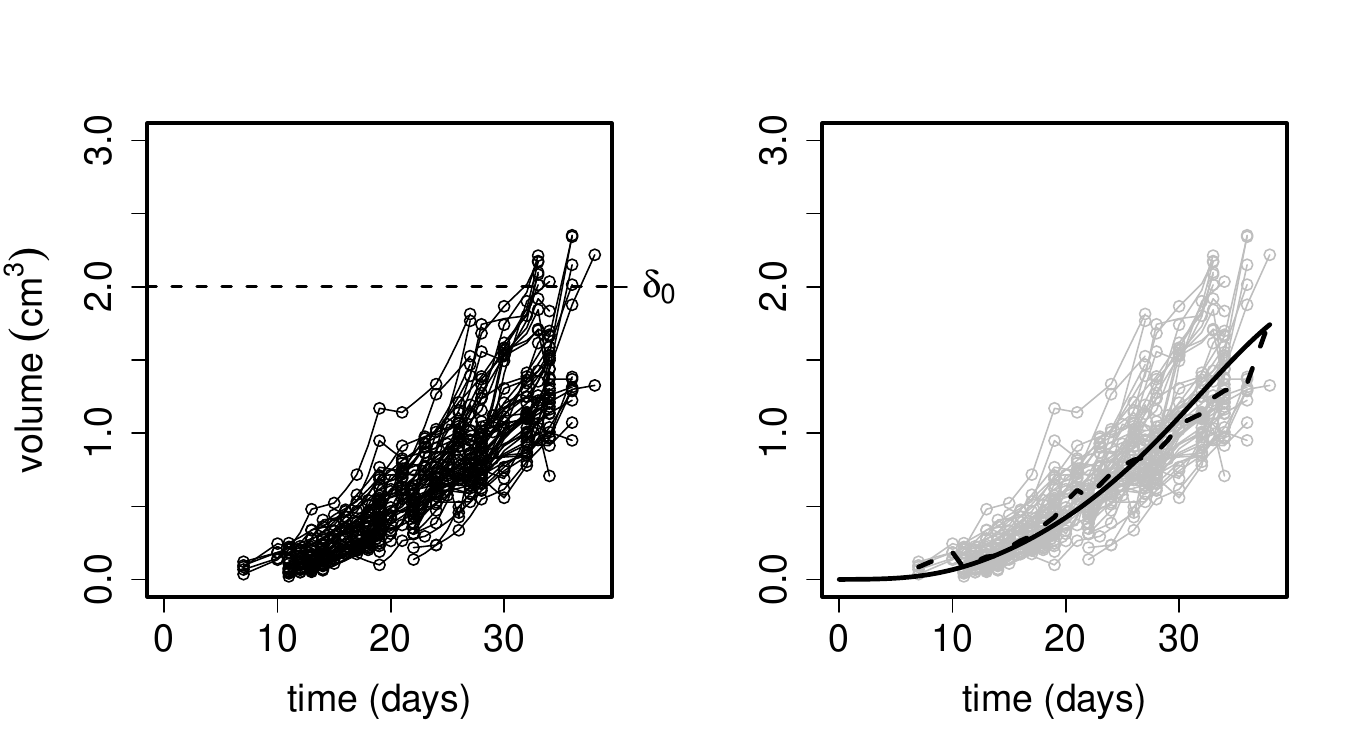}
    \caption{Tumor growth curves of $N = 66$ mice obtained from different animal studies. Data collection terminates when a subject's tumor volume exceeds the threshold of $\delta_0 = 2\text{\;cm}^3$. The right plot shows the estimated median function under the proposed model (solid) together with a naive point-wise estimate that ignores missing data (dashed).}
    \label{fig:tumor_growth}
\end{figure}

To illustrate the proposed methodology, we analyze a secondary data set of tumor growth compiled by \cite{vaghi2020population}. The data originate from eight studies conducted in accordance with established guidelines for the care and use of laboratory animals (see the ethics statement in \cite{vaghi2020population}). In each study, LM2-4$^\text{LUC+}$ breast tumor cells were injected into the right inguinal mammary fat pads of $N = 66$ mice. Tumor volumes were subsequently monitored for up to $T=38$ days or until the tumor volume reached the ethical limit of $2\;\text{cm}^3$. Tumor size was measured using calipers at irregular, subject-specific observation times. 

Let $Y_i = (Y_i(t) \colon t \in [0,T])$ denote the tumor growth trajectory for the $i$-th mouse and let $0 \leq t_{i1} < \dots < t_{in_i} \leq T$ be its observation times. The observed data for subject $i$ therefore consist of the pairs $(t_{ij},Y_i(t_{ij}))$, $j=1,\ldots,n_i$. Figure~\ref{fig:tumor_growth} displays the observed trajectories. Measurements are discontinued once the tumor volume exceeds the ethical threshold of $2\text{\;cm}^3$, giving rise to an ODD mechanism. Because the decision to terminate observations depends only on the currently observed tumor volume, the resulting missingness mechanism satisfies Rubin's \emph{missing-at-random} (MAR) assumption. Accounting for this mechanism is essential for valid statistical inference: as illustrated in the right panel of Figure~\ref{fig:tumor_growth}, the naive pointwise median systematically underestimates the true median toward the end of the observation period. One of the principal advantages of the proposed likelihood-based framework is that it naturally accommodates this MAR sampling mechanism, whereas existing moment-based procedures generally rely on the stronger missing-completely-at-random (MCAR) assumption and consequently suffer from bias 
under ODD. 

The rest of the paper is organized as follows. Section~\ref{sec:GMProcesses} introduces the Gauss-Markov framework as the distributional model of the underlying stochastic processes. 
Section~\ref{sec:Measurements} presents details on the sparse sampling regime and on the outcome-dependent dropout mechanism. Section~\ref{sec:Estimation} develops the proposed likelihood estimation procedure, whose theoretical properties in terms of convergence rates are investigated in Section~\ref{sec:ConvergenceRates}. We present finite-sample properties and a detailed analysis of the tumor growth data in Sections~\ref{sec:NumericalIllustration} and \ref{sec:RealDataIllustration}. Section~\ref{sec:Conclusions} concludes. Proofs of the theoretical results are provided in a technical Appendix.

\section{Gauss-Markov processes}
\label{sec:GMProcesses}

Consider a generic stochastic process $X=(X(t)\colon t\in[0,T])$, representing a single observation in a functional random sample. The objective of a distributional model is to characterize the probabilistic evolution of $X$ over time. Such a model should be sufficiently flexible to capture the temporal dependence of the process while remaining amenable to statistical inference.

A natural starting point is to view the observations as a discretized time series,
\[
Y_j=X(t_j), \qquad j=1,\ldots,n,
\]
and to model these data using a classical time series model. The simplest example is the first-order autoregressive model, AR(1). Passing from discrete to continuous time leads to its continuous-time analogue, the continuous autoregressive model CAR(1), which is equivalent to the Ornstein--Uhlenbeck process. More generally, the Ornstein--Uhlenbeck process belongs to the class of Gauss--Markov processes, obtained as solutions of linear SDEs. These processes provide a parsimonious yet flexible family of distributional models and have been used extensively in fields such as statistical physics, engineering, finance, and the environmental sciences. Their tractability and probabilistic structure make them particularly attractive for modeling functional data.

\subsection{Definition}
\label{subsec:GM-Def}

Let $b_1,b_2\colon [0,T]\to\mathbb{R}$ be continuous drift functions, let $\sigma\colon [0,T]\to\mathbb{R}$ be a continuous diffusion function, and let $B=(B(t)\colon t\ge0)$ denote a standard Brownian motion. We consider stochastic processes
\[
X=(X(t)\colon t\in[0,T])
\]
that satisfy the linear stochastic differential equation
\begin{equation}\label{eq:SDE}
\begin{cases}
\diff X(t)
=
\bigl(b_1(t)X(t)+b_2(t)\bigr)\,\diff t
+
\sigma(t)\,\diff B(t),
&
t\in[0,T],\\[0.3em]
X(0)=X_0,
\end{cases}
\end{equation}
where the initial value $X_0\sim N(m_0,v_0)$ is independent of $B$. The special case of a deterministic initial condition, $X(0)=m_0$, is obtained by setting $v_0=0$.

Equation~\eqref{eq:SDE} is understood in the integral sense,
\[
X(t)
=
X_0
+
\int_0^t
\bigl(b_1(u)X(u)+b_2(u)\bigr)\,\diff u
+
\int_0^t
\sigma(u)\,\diff B(u),
\qquad t\in[0,T],
\]
where the second integral is an Itô integral. Standard results guarantee the existence of a unique solution under these assumptions. For general introductions to stochastic differential equations, we refer to
\citet{karatzas1991brownian},
\citet{kloeden1992numerical}, and
\citet{oksendal2007stochastic}.

\subsection{Distribution}
\label{subsec:GM-Dist}

The linear SDE~\eqref{eq:SDE} admits an explicit solution. To derive it, consider the associated homogeneous ordinary differential equation
\begin{equation}\label{eq:hLSDE}
\dot{x}(t)=b_1(t)x(t), \qquad x(0)=1.
\end{equation}
Its solution is given by
\[
x(t)=\Phi(0,t),
\]
where
\begin{equation}\label{eq:FundamentalSolution}
\Phi(s,t)
=
\exp\!\left(\int_s^t b_1(u)\,\diff u\right),
\qquad
s,t\in[0,T],
\end{equation}
is the corresponding \emph{fundamental solution}. Applying the variation-of-constants formula yields
\begin{equation}\label{eq:SolutionLSDE}
X(t)
=
\Phi(s,t)X(s)
+
\int_s^t \Phi(u,t)b_2(u)\,\diff u
+
\int_s^t \Phi(u,t)\sigma(u)\,\diff B(u),
\end{equation}
for every \(0\le s<t\le T\).

Since~\eqref{eq:SolutionLSDE} is an affine transformation of jointly Gaussian random variables, the process \(X\) is Gaussian. Moreover, it satisfies the strong Markov property. Throughout the paper we therefore refer to \(X\) as a \emph{Gauss--Markov (GM) process}. Its mean function is given by
\[
\mu(t)
=
\expv[X(t)]
=
\Phi(0,t)m_0
+
\int_0^t
\Phi(u,t)b_2(u)\,\diff u,
\qquad
t\in[0,T].
\]

The following proposition summarizes the unconditional and conditional moments of the process.

\begin{proposition}\label{prop:CondMeanVar}
Let $X$ be the solution of~\eqref{eq:SDE}. Then, for any
$0\le s\le t\le T$ and $x\in\mathbb{R}$,
\begin{align*}
\expv[X(t)]
&=
\Phi(0,t)m_0
+
\int_0^t
\Phi(u,t)b_2(u)\,\diff u,\\[0.3em]
\expv[X(t)\vert X(s)=x]
&=
\Phi(s,t)x
+
\int_s^t
\Phi(u,t)b_2(u)\,\diff u,\\[0.3em]
\cov(X(s),X(t))
&=
\Phi(0,s)\Phi(0,t)v_0
+
\int_0^s
\Phi(u,s)\Phi(u,t)\sigma^2(u)\,\diff u,\\[0.3em]
\var(X(t)\vert X(s)=x)
&=
\int_s^t
\Phi^2(u,t)\sigma^2(u)\,\diff u.
\end{align*}
\end{proposition}

For later use, define
\[
m(x,s,t)
=
\expv[X(t)\vert X(s)=x],
\qquad
v(s,t)
=
\var(X(t)\vert X(s)=x).
\]

The expressions in Proposition~\ref{prop:CondMeanVar} reflect a fundamental property of Gaussian processes: the conditional distribution of $X(t)$ given $X(s)=x$ is Gaussian, with a conditional mean that is affine in $x$ and a conditional variance that depends only on the time points $s$ and $t$.

\subsection{Doob representation}
\label{subsec:GM-Doob}

An alternative characterization of Gauss--Markov processes is provided by the Doob representation
\begin{equation}\label{eq:Doob}
    X(t) = \mu(t) + h_0(t) B(h_1(t)),\qquad t\in [0,T],
\end{equation}
for suitable deterministic functions~$h_0, h_1\colon [0,T] \to \mathbb{R}$; see also \cite{mehr1965certain}. Conversely, provided that the functions are sufficiently smooth, Itô's formula \cite[Theorem~4.1.2]{oksendal2007stochastic} implies that a process defined by \eqref{eq:Doob} satisfies a linear SDE of the form~\eqref{eq:SDE}. Thus, the Doob representation and the linear SDE formulation are equivalent descriptions of a broad class of Gauss--Markov processes.

The representation is particularly useful because it separates the deterministic and stochastic components of the process. We will exploit this feature in Section~\ref{sec:RealDataIllustration} to construct simultaneous prediction bands. The next proposition expresses the functions $h_0$ and $h_1$ in terms of the parameters of the SDE.

\begin{proposition}\label{prop:Doob}
    Let $X$ be the solution of SDE~\eqref{eq:SDE}. The functions $h_0$ and $h_1$ in Doob's representation~\eqref{eq:Doob} are then given by
    \begin{itemize}
        \item $h_0(t) = \Phi(0, t)$,
        \item $h_1(t) = v_0 + \int_0^t \Phi(0, u)^{-2} \sigma^2(u) \, \diff u$,
    \end{itemize}
    for any $0 \leq t \leq T$.
\end{proposition}
\noindent Throughout the paper, $h_0$ is chosen to be positive. Furthermore, $h_1$ is strictly increasing, with $h_1(0)=v_0$, the variance of the initial condition. Finally, note that the pair $(h_0,h_1)$ is identifiable only up to multiplication of $h_0$ and reciprocal scaling of $h_1$ by a positive constant.

\section{Sparse sampling and outcome-dependent dropout}
\label{sec:Measurements}

Let $X_1,\ldots,X_N$ be independent copies of the Gauss--Markov process $X$. Throughout the paper, we assume that the trajectories are observed sparsely and are subject to outcome-dependent dropout. The sparse sampling scheme is introduced in Section~\ref{sec:SparseDesign}, while the dropout mechanism is described in Section~\ref{sec:MissingData}.

\subsection{Sparse sampling}\label{sec:SparseDesign}

For each subject $i=1,\ldots,N$, observations are collected at subject-specific time points,
\[
X_{ij}=X_i(T_{ij}),
\qquad
j=1,\ldots,n_i,
\]
where
\begin{enumerate}[label=(\alph*)]
\item $N$ denotes the sample size;
\item $n_i$ is the number of design points for subject $i$;
\item the design points $T_{ij}$ are independent across subjects, uniformly distributed on $[0,T]$, and ordered ($0\le T_{i1}<\cdots<T_{in_i}\le T$);
\item $T_{i1},\ldots,T_{in_i}$ and $X_i$ are independent.
\end{enumerate}
Assumptions (a)--(d) are standard in functional data analysis under sparse designs; see, for example, \citet{yao2005functional} and \citet{li2010uniform}. In particular, random observation times that are independent of the underlying process, with a sampling density bounded away from zero on the observation interval, constitute a commonly used sparse random-design framework. For simplicity, we do not include an additional measurement-error term. Any short-term variability is assumed to be captured by the diffusion component of the SDE.

\subsection{Outcome-dependent dropout}\label{sec:MissingData}

A principal advantage of the proposed likelihood framework is that it naturally accommodates ODD, as illustrated by the tumor-growth data discussed in Section~\ref{sec:intro}. Throughout the paper we consider the following sampling mechanism.

\begin{equation*}
\parbox{0.83\textwidth}{
There exists a threshold \(\delta_0\in\mathbb{R}\) such that $(T_{ij},X_{ij})$ is observed if and only if
\begin{equation}
\max\{X_{i1},\ldots,X_{i,j-1}\}\le\delta_0.
\tag{M}\label{M}
\end{equation}
}
\end{equation*}

Thus, observations on a subject cease immediately after the process first exceeds the threshold~$\delta_0$. In the tumor-growth application, for example, measurements stop once the tumor volume exceeds $2\mathrm{\;cm}^3$. The methodology developed below does not require knowledge of the threshold value and therefore also applies when $\delta_0$ is unknown. Complete observations correspond to the limiting case $\delta_0=\infty$.

Note that mechanism \eqref{M} does not contradict the independence assumption on the sampling design in (d) above. The scheduled observation times are generated independently of the process before any measurements are taken, whereas the realized number of observations may be smaller because the sampling protocol terminates once the threshold is crossed.

\subsection{Inference under outcome-dependent dropout}
\label{sec:ConsequencesOfMissingData}

Most statistical procedures for partially observed functional data assume that the missingness mechanism is MCAR, meaning that the observation pattern is independent of the underlying trajectory. Mechanism (M) clearly violates this assumption because the probability of observing future measurements depends on previously observed values of the process. Nevertheless, as the following proposition shows, it satisfies the weaker MAR condition. Let $O\subset\{T_1,\ldots,T_n\}$ denote the random set of observed measurement locations, and let $\mathcal{F}_K=\sigma\{X(t)\colon t\in K\}$, $K\subset[0,T]$, be the $\sigma$-algebra generated by $X$ on the set $K$.

\begin{proposition}\label{prop:MAR}
Under the outcome-dependent dropout mechanism \eqref{M}, the observation process satisfies the missing-at-random (MAR) property. Specifically,
\[
\prob(O\subset K\vert\mathcal{F}_{[0,T]})
=
\prob(O\subset K\vert\mathcal{F}_K),
\]
for every measurable set $K\subset[0,T]$.
\end{proposition}
The MAR framework was introduced by \citet{rubin1976inference} and has since been refined by \citet{seaman2013meant} and \citet{farewell2022missing}. A functional-data analogue has recently been proposed by \citet{ofner2025testing}. 

Proposition~\ref{prop:MAR} has important implications for statistical inference. Under the ODD mechanism~(M), moment-based estimators are generally inconsistent because they are computed from the observed sample rather than the target population. Consequently, sample estimates of the mean and covariance converge to conditional rather than unconditional moments. In contrast, the MAR property ensures that likelihood-based inference remains valid without explicitly modeling the missingness mechanism (see \cite{rubin1976inference,seaman2013meant}). This observation motivates the likelihood-based estimation procedure developed in Section~\ref{sec:Estimation}. Appendix~\ref{app:A} provides simple examples illustrating the contrasting behavior of moment-based and likelihood-based estimators.

\section{Estimation}
\label{sec:Estimation}

The MAR property established in the previous section provides the basis for likelihood-based inference. We now derive a maximum likelihood estimator for the parameters of the Gauss--Markov process X defined by the SDE~\eqref{eq:SDE}. Let $f(x_{i1},\ldots,x_{in_i},t_{i1},\ldots,t_{in_i})$
denote the joint density of the observations from subject $i$. Since the sampling design is independent of the underlying process,
\[
f(x_{i1},\ldots,x_{in_i},t_{i1},\ldots,t_{in_i})
=
f_{X_i(t_{i1}),\ldots,X_i(t_{in_i})}
(x_{i1},\ldots,x_{in_i})
\,
f_{T_{i1},\ldots,T_{in_i}}
(t_{i1},\ldots,t_{in_i}).
\]
The second factor does not involve the model parameters and can therefore be omitted from the likelihood. Consequently, inference may be based solely on the conditional density of the observed process. By the Markov property,
\begin{equation}\label{eq:Density}
f_{X_i(t_{i1}),\ldots,X_i(t_{in_i})}
=
f_{X_i(t_{i1})}
\prod_{j=2}^{n_i}
f_{X_i(t_{ij})\,|\,X_i(t_{i,j-1})},
\end{equation}
so that the likelihood factorizes into a product of Gaussian transition densities.

\subsection{Likelihood}

Conditional on the measurement locations, the $(-2)$-log-likelihood of the complete data is
\begin{equation*}
\mathcal{L}(b_1, b_2, \sigma, m_0, v_0) = \sum_{i=1}^N \sum_{j=1}^{n_i}\left \lbrace \log(2\pi) + \log(v_{ij}) + \frac{(x_{ij}-m_{ij})^2}{v_{ij}}\right\rbrace,
\end{equation*}
where $m_{i1} = \Phi(0,t_{i1}) \, m_0 + m(0,0,t_{i1})$, $v_{i1} = \Phi(0,t_{i1})^2\, v_0 + v(0, t_{i1})$, as well as
\begin{align*}
m_{ij} = m(x_{i,j-1}, t_{i,j-1}, t_{ij}),
\qquad v_{ij} = v(t_{i,j-1}, t_{ij}),
\qquad j = 2, \dots, n_i.
\end{align*}
The dependence of $m_{ij}$ and $v_{ij}$ on the model parameters is discussed in Proposition~\ref{prop:CondMeanVar}. Under the dropout mechanism~\eqref{M}, the complete likelihood is not observable because some responses are missing. Since the missingness mechanism is MAR, inference can instead be based on the observed likelihood
\begin{equation}\label{eq:Logl}
\mathcal{L}_\text{obs}(b_1, b_2, \sigma, m_0, v_0) = \sum_{i=1}^N \sum_{j=1}^{n_i}\left \lbrace \log(2\pi) + \log(v_{ij}) + \frac{(x_{ij}-m_{ij})^2}{v_{ij}}\right\rbrace \ind\{\text{$x_{ij}$ obs}\}.
\end{equation}
Unlike the complete likelihood, $\mathcal{L}_\text{obs}$ depends only on observed measurements and is therefore directly computable from the available data.

\subsection{Sieve maximum likelihood estimation}

The drift and diffusion functions are infinite-dimensional parameters and therefore cannot be estimated directly by finite-dimensional optimization. We adopt a sieve maximum likelihood approach and approximate these functions by finite-dimensional basis expansions. Specifically, let $M=M(N)$ be a truncation parameter satisfying $M\to\infty$ as $N\to\infty$. As in other sieve methods, the choice of $M$ balances approximation bias against estimation variance.

Consider the cosine basis $\varphi_1(t) = 1/\sqrt{T}$, $\varphi_{k+1}(t) = \sqrt{2/T} \cos(k \pi t/T)$ for $k \geq 1$, and the expansions
\begin{equation}\label{eq:Expansion}
\begin{aligned}
        b_1(t) = \sum_{k=1}^\infty  b_{1k} \varphi_k(t), && b_2(t) = \sum_{k=1}^\infty  b_{2k} \varphi_k(t), &&\sigma(t) = \exp\left(\sum_{k=1}^\infty  \sigma_{k} \varphi_k(t)\right).
\end{aligned}
\end{equation}
Since the drift and diffusion functions are assumed to be smooth on $[0,T]$, they can be approximated arbitrarily well by finite linear combinations of the cosine basis functions. This motivates replacing the infinite-dimensional estimation problem by a sequence of finite-dimensional approximation spaces whose dimension increases with the sample size. To this end, note that the sieve approximations corresponding to \eqref{eq:Expansion} are 
\begin{equation}\label{eq:ExpansionTruncated}
\begin{aligned}
        b_{1N}(t) = \sum_{k=1}^M  b_{1k} \varphi_k(t), && b_{2N}(t) = \sum_{k=1}^M b_{2k} \varphi_k(t), &&\sigma_{N}(t) = \exp\left(\sum_{k=1}^M  \sigma_{k} \varphi_k(t)\right),
\end{aligned}
\end{equation}
for $t \in [0,T]$. The exponential parameterization guarantees that $\sigma_N$ is non-negative. The sieve maximum likelihood estimator is obtained by minimizing the observed likelihood $\mathcal{L}_\text{obs}$ in~\eqref{eq:Logl},
\begin{equation}\label{eq:OptProblem}
\big(\widehat{b}_{1N}, \widehat{b}_{2N}, \widehat{\sigma}_N, \widehat{m}_0, \widehat{v}_0\big) = \argmin_{\substack{b_{1N}, b_{2N}, \sigma_{N}, \\[0.5ex]m_0, \, v_0}}\, \mathcal{L}_\text{obs}(b_{1N}, b_{2N}, \sigma_N, m_0, v_0).
\end{equation}

The optimization problem in \eqref{eq:OptProblem} is finite-dimensional, with parameter dimension $d=3M+2$, which increases with the sample size. This approach is known as sieve maximum likelihood estimation; see Section 3.4 of \citet{van2023weak}. Sieve methods provide a principled framework for estimating infinite-dimensional parameters by approximating them with a sequence of finite-dimensional models whose complexity grows with the sample size. The proposed method  thus also links the statistical assumption of smoothness and the optimization method using the cosine sieve. 

\begin{remark}\label{rem:ChoiceM}
For notational simplicity, all three functional parameters are approximated using the same truncation level $M$. In practice, however, it is often advantageous to choose separate truncation levels $(M_{b_1},M_{b_2},M_\sigma)$ for the drift and diffusion functions. This modification is straightforward, as it just needs notational adjustments, and is adopted in the numerical implementation of Section~\ref{sec:NumericalIllustration}.
\end{remark}

\section{Convergence rates}\label{sec:ConvergenceRates}

This section establishes convergence rates for the proposed sieve maximum likelihood estimator. To simplify the asymptotic analysis, we study an equivalent $M$-estimator based on conditional transition densities. This formulation retains the essential statistical structure of the likelihood, while avoiding technical complications arising from the initial distribution and additive constants.

\subsection{Assumptions}

Let $[0,T] = [0,1]$ and impose the following assumptions throughout this section.
\begin{assumption}\label{ass:X}
The process $X$ solves the SDE~\eqref{eq:SDE} with $X(0)\sim{}  N(0,v_0)$ for some $v_0>0$.
\end{assumption}
\noindent Assumption~\ref{ass:X} specifies that the initial distribution has a positive variance. This assumption is imposed solely to simplify the exposition; the extension to a general Gaussian initial distribution is feasible.

\begin{assumption}\label{ass:ni}
Each subject is measured at $n_i = 2$ locations according to \eqref{M}. 
\end{assumption}

The restriction to two scheduled measurements per subject represents the sparsest sampling design for which the transition density is identifiable. It is adopted for technical convenience and is not essential for the methodology. The second observation remains subject to the dropout mechanism~\eqref{M}, so the asymptotic analysis continues to capture the ODD effect on the likelihood estimation procedure. 

\begin{assumption}\label{ass:beta}
There exist constants $C>0$ and $\beta > 1$ such that the coefficients in~\eqref{eq:Expansion} satisfy $\max\{\abs{b_{1k}}, \abs{b_{2k}}, \abs{\sigma_k}\} \leq C k^{-\beta}$ for all $k \geq 1$.
\end{assumption}
\noindent Assumption~\ref{ass:beta} is a standard smoothness condition requiring the Fourier coefficients of the drift and diffusion functions to decay polynomially. It guarantees that the coefficient functions are uniformly bounded and that the approximation error of the sieve decreases at a polynomial rate.

We then consider the objective function
\begin{equation}\label{eq:SLogl}
S_N(b_1, b_2, \sigma) = \frac{1}{N}\sum_{i=1}^N \left\lbrace-\log(v_{i2}) - \frac{(x_{i2}-m_{i2})^2}{v_{i2}}\right\rbrace \ind\{x_{i1} \leq \delta_0\},
\end{equation}
and introduce the $M$-estimator,
\begin{equation}\label{eq:EstProblem}
\big(\widehat{b}_{1N}, \widehat{b}_{2N}, \widehat{\sigma}_{N}\big) = \argmax_{\substack{b_{1N}, b_{2N}, \sigma_{N}}}\, S_N(b_{1N}, b_{2N}, \sigma_N).
\end{equation}

\begin{table}
    \centering
    \begin{tabular}{|l|c|}
    \hline
    \textit{Notation} & \textit{Definition}\\
     \hline
    $\normww{f} $   &\rule[-3ex]{0pt}{8ex} $\sqrt{\int_0^1 \int_0^{t_2} \frac{1}{t_2 - t_1} \left(\int_{t_1}^{t_2} f(u) \, \diff u\right)^2 \,\diff t_1\, \diff t_2}$ \\
     \hline
     $\normwww{f} $   &\rule[-3ex]{0pt}{8ex} $\sqrt{\int_0^1 \int_0^{t_2} \frac{1}{(t_2 - t_1)^2} \left(\int_{t_1}^{t_2} f(u) \, \diff u\right)^2 \,\diff t_1\, \diff t_2}$ \\
     \hline
    $\norm{f}_{L^2}$  &\rule[-2ex]{0pt}{6ex} $\sqrt{\int_0^1 f(u)^2 \, \diff u}$ \\
         \hline
     $\norm{f}_\infty$  & \rule[-2ex]{0pt}{5ex}$\sup_{u \in [0,1]} \, \abs{f(u)}$ \\
         \hline
    \end{tabular}
    \caption{Weighted norms used in the convergence analysis. The first two norms arise naturally from the likelihood and are ordered from weaker to stronger, followed by the classical $L^2$ and supremum norms for comparison.}
    \label{tab:Norms}
\end{table}

\subsection{Results}

The likelihood and its derived $M$-estimator depend on the unknown coefficient functions only through the (Gaussian) transition means and variances, which in turn involve integral functionals of the drift and diffusion coefficients. Consequently, a local quadratic approximation of the Kullback-Leibler divergence induces weighted quadratic forms different from the classical $L^2$-norm. These norms, introduced in Table~\ref{tab:Norms}, therefore arise naturally in the asymptotic analysis. They satisfy
 \begin{equation*}
\normww{f} \leq \normwww{f} \leq \norm{f}_{L^2} \leq \norm{f}_\infty.
\end{equation*}
Tighter bounds and a reversed inequality for the truncated expansions are discussed in Section~\ref{sec:Norms} of the Appendix.

\begin{theorem}[Convergence rates in weighted norms]\label{thm:ConvergenceRate}
Suppose that Assumptions~\ref{ass:X}--\ref{ass:beta} are satisfied and let $(b_{10}, b_{20}, \sigma_0)$ be the true parameters. Set $M \sim N^{1/2\beta}$. Then,
\begin{equation*}
    \normw{\widehat{b}_{1N} - b_{10}}^2 +\normww{\widehat{b}_{2N} - b_{20}}^2 + \normwww{\widehat{\sigma}_N^2-\sigma_{0}^2}^2 = O_p\left(\log(N) \,N^{-(2\beta-1)/2\beta}\right).
\end{equation*}
as $N \to \infty$.
\end{theorem}
\noindent The proof of Theorem~\ref{thm:ConvergenceRate} is given in Section~\ref{sec:MainProofs} of the Appendix. The convergence rate depends on the smoothness parameter $\beta$, measured through the decay of the Fourier coefficients, as is typical in nonparametric function estimation and ill-posed inverse problems; see \citet{hall2007methodology} and \citet{kato2012estimation}. As smoothness increases ($\beta\to\infty$), the convergence rate approaches the parametric rate $N^{-1}$. 

Theorem~\ref{thm:ConvergenceRate} is formulated in weighted norms that reflect the intrinsic difficulty of estimating the three coefficient functions. Since $\normww{f} \leq \normwww{f}$, the strongest convergence guarantee is obtained for the diffusion coefficient $\sigma^2$, whereas the convergence guarantee obtained for the drift coefficients $b_1$ and $b_2$ is weaker. 

Since the intrinsic norms in Theorem~\ref{thm:ConvergenceRate} are tailored to the likelihood geometry, it is also useful to derive convergence rates in the more familiar $L^2$-norm, which provides a direct and interpretable measure of estimation accuracy for the coefficient functions and facilitates comparison with existing nonparametric results. This is done in the next corollary, which follows directly from the common weighted norms appearing in Theorem~\ref{thm:ConvergenceRate}, together with the reversed inequalities in Section~\ref{sec:Norms} of the Appendix.

\begin{corollary}[Convergence rates in the $L^2$-norm]\label{cor:L2Rates}
If the assumptions of Theorem~\ref{thm:ConvergenceRate} are satisfied,  the following $L^2$ convergence rates hold. Choose $M\sim N^{1/(2\beta+2)}$, then 
\begin{align*}
\|\widehat b_{1N}-b_{10}\|_{L^2}^2
&=
O_p\!\left(
\log(N)\,
N^{-(2\beta-1)/(2\beta+2)}
\right),\\
\|\widehat b_{2N}-b_{20}\|_{L^2}^2
&=
O_p\!\left(
\log(N)\,
N^{-(2\beta-1)/(2\beta+2)}
\right).
\end{align*}
Choose $M\sim N^{1/(2\beta+1)}$, then
\begin{align*}
\|\widehat\sigma_N^2-\sigma_0^2\|_{L^2}^2
&=
O_p\!\left(
\log(N)\,
N^{-(2\beta-1)/(2\beta+1)}
\right).
\end{align*}
\end{corollary}

Let $\mathcal{P} = \mathcal{P}(C,\beta)$ denote the set of distributions of $(X_{i1}, X_{i2}, T_{i1}, T_{i2}\colon i \leq N)$ that are compatible with Assumptions~1--3. The following result shows that the polynomial convergence rates of the proposed estimators are optimal in a minimax sense, up to polylog factors.

\begin{theorem}[Minimax optimality in the $L^2$-norm]\label{thm:L2}
    Suppose the assumptions of Theorem~\ref{thm:ConvergenceRate} are satisfied. Then,
    \begin{align*}
        &\liminf_{N \to \infty}\inf_{\widehat{b}_{1N}} \sup_{P\in \mathcal{P}} \prob_P \left(\norm{\widehat{b}_{1N} - b_{10}}_{L^2}^2 \geq (N \log N)^{-(2\beta-1)/(2\beta + 2)}\right) > 0,\\
       &\liminf_{N \to \infty}\inf_{\widehat{b}_{2N}} \sup_{P\in \mathcal{P}} \prob_P \left(\norm{\widehat{b}_{2N} - b_{20}}_{L^2}^2 \geq (N \log N)^{-(2\beta-1)/(2\beta + 2)}\right) > 0,\\
       &\liminf_{N \to \infty}\inf_{\widehat{\sigma}_N^2}\sup_{P\in \mathcal{P}} \prob_P \left(\norm{\widehat{\sigma}_N^2 - \sigma_0^2}_{L^2}^2 \geq N^{-(2\beta-1)/(2\beta + 1)}\right) > 0.
    \end{align*}
\end{theorem}

The proof of Theorem~\ref{thm:L2} is given in Section~\ref{sec:MainProofs} of the Appendix. A notable implication of the theorem is that the diffusion coefficient can be estimated at a strictly faster minimax rate than the drift functions. While the $L^2$-error for the drift coefficients has polynomial rate $N^{-(2\beta-1)/(2\beta+2)}$, up to logarithmic factors, the corresponding rate for $\sigma^2$ is $N^{-(2\beta-1)/(2\beta+1)}$. The difference reflects the distinct ways in which drift and diffusion enter the transition distribution of the SDE: the drift affects the conditional mean through integrated coefficient functions, whereas the diffusion is identified through the conditional variance and is therefore subject to a lower degree of ill-posedness. This phenomenon parallels the continuous-observation setting, where the diffusion coefficient is directly identified through the quadratic variation of the process and can typically be estimated more accurately than the drift; see, for example, \citet{kutoyants2004statistical} and \citet{jacod2003limit}. Our minimax result shows that this distinction persists under the sparse sampling scheme considered here and that it is intrinsic to the statistical problem rather than an artifact of the proposed estimation procedure.

\begin{remark}
    Assumption~\ref{ass:X} requires a nondegenerate initial distribution with variance~$v_0>0$. This ensures that the process exhibits nonvanishing variation near the origin, causing the two drift components~$b_1$ and $b_2$ to have the same degree of ill-posedness and leading to the same polynomial $L^2$-rate. The deterministic initial condition $X(0)=0$ results in a different convergence rate for the estimator of $b_1$. The vanishing variation of the SDE as $t\to 0$ leads to a loss of information for the estimation of the multiplicative drift $b_1$ and our methodology yields a convergence rate of~polynomial order $N^{-(2\beta-1)/(2\beta+3)}$. The corresponding rates for $b_2$ and $\sigma^2$ are unaffected by the initial condition.
\end{remark}

\section{Finite sample properties}\label{sec:NumericalIllustration}

We investigate the finite-sample performance of the proposed methodology through two simulation studies. The first compares the likelihood-based estimator with an existing moment-based procedure, while the second considers a data-generating process designed to resemble the tumor-growth application.

\subsection{Simulation design and implementation}\label{sec:Implementation}

Using the statistical programming language R \citep{R} for the implementation, we generate independent copies of the process $X=(X(t)\colon t\in[0,1])$ satisfying
\[
\diff X(t)
=
(b_1(t)X(t)+b_2(t))\,\diff t+ \sigma(t)\,\diff B(t), \quad t\in[0,1];
\qquad X(0)\sim X_0,
\]
where $X_0\sim N(m_0,v_0)$ is independent of the Brownian motion $B$. The nonlinear optimization problem~\eqref{eq:OptProblem} is solved using the Newton-type algorithm \emph{nlm} from the \emph{stats} package. Following Remark~\ref{rem:ChoiceM}, separate truncation parameters are selected for $b_1$, $b_2$, and $\sigma$ using forward AIC selection. Specifically, we minimize
\[
\text{AIC}(M_{b_1},M_{b_2},M_\sigma)
=
\widehat{\mathcal L}_{\text{obs}}
+
2(M_{b_1}+M_{b_2}+M_{\sigma}+2).
\]
Starting from $(M_{b_1},M_{b_2},M_\sigma)=(0,0,0)$, the procedure successively increases the truncation parameter yielding the largest reduction in AIC and terminates when no further improvement is possible or $M_{\max}=10$ is reached. The implementation is available from the third author's GitHub repository\footnote{See \url{https://github.com/maxofn}.}.

 We consider the two parameter configurations specified below. For each configuration, we generate $1{,}000$ samples of size $N\in\{100,200,500\}$, with each trajectory scheduled to be observed at $n\in\{5,10\}$ uniformly distributed time points. Observations are subject to mechanism~\eqref{M}; the choice $\delta_0=\infty$ corresponds to complete sparse observations.

Estimation accuracy for a functional parameter $\theta$ is measured by the root integrated squared error
\[
\text{RISE}(\widehat\theta,\theta)
=
\left(
\int_0^1
(\widehat\theta(t)-\theta(t))^2\,\diff t
\right)^{1/2}.
\]
For the scalar parameters $m_0$ and $v_0$, we use absolute error. Reported results are means and standard deviations across the $1{,}000$ simulation runs.

\subsection{Comparison with nonparametric benchmark}

\textit{Data-generating process.}
The first configuration is
\begin{equation}\tag{A}\label{config:A}
    \begin{aligned}
        b_1(t) &= \frac{1 - e^{10t - 5}}{1 + e^{10t - 5}}, & \qquad \sigma(t) &= \sqrt{\exp((1-t)^2)}, \\  
        b_2(t) &= 0, & (m_0, v_0) &= (2, 0).
    \end{aligned}
\end{equation}
Thus, $X(0)=2$ is deterministic. We consider both complete observations $(\delta_0=\infty)$ and outcome-dependent dropout with $\delta_0=3$. Figure~\ref{fig:Samples} illustrates the resulting sampling schemes for $(N,n)=(500,5)$.

\begin{figure}
    \centering
    \caption*{\textsc{Configuration}~\eqref{config:A}}
    \begin{subfigure}[b]{0.47\textwidth}
        \centering
        \caption*{(\textsc{Complete}: $\delta_0 = \infty$)}
         \vspace{-0.5cm}
    \includegraphics[width=0.99\textwidth]{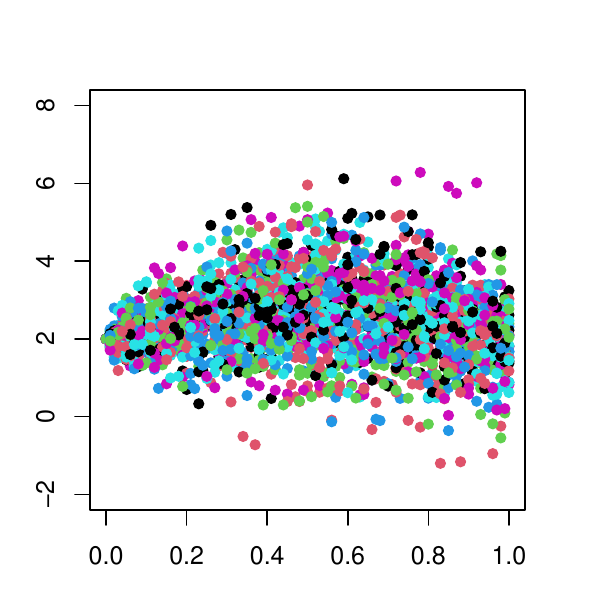}
    \end{subfigure}
    \hspace{0.5cm}
    \begin{subfigure}[b]{0.47\textwidth}
        \centering
        \caption*{(\textsc{MAR}: $\delta_0 = 3$)}
        \vspace{-0.5cm}
    \includegraphics[width=0.99\textwidth]{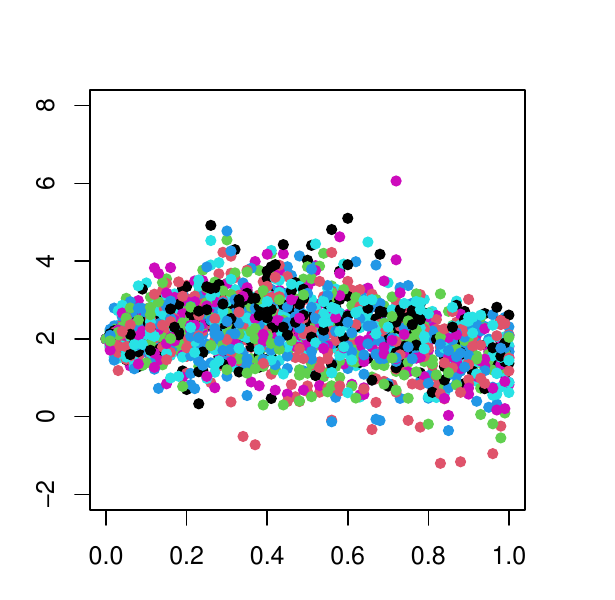}
    \end{subfigure}
    \caption{Samples of $N = 500$ trajectories recorded at $n = 5$ sparse measurement times, with colors indicating individual subjects. The left panel shows complete observations, the right panel illustrates ODD with threshold.}\label{fig:Samples}
\end{figure}

\medskip
\noindent \textit{Benchmark.}
We compare the proposed estimator with the nonparametric method of \citet{mohammadi2024nonparametric}. For the model
$
\diff X(t)=b_1(t)X(t)\,\diff t+\sigma(t)\,\diff B(t),
$
let $\mu(t)=\expv[X(t)]$ and $\nu(t)=\expv[X^2(t)]$. These functions satisfy the ordinary differential equations (ODEs)
$\dot\mu=b_1\mu$ and
$\dot\nu=2b_1\nu+\sigma^2$, hence
$b_1=\dot\mu/{\mu}$ and
$\sigma^2=\dot\nu-2b_1\nu$. We consider their ODE-based version, which first estimates the moments and their derivatives by local polynomial smoothing, and plugs them into the above identities to estimate $b_1$ and $\sigma^2$. Denote the resulting estimators by $\widehat b_{1,\mathrm{MSP}}$ and $\widehat\sigma^2_{\mathrm{MSP}}$. Although $b_2=0$ under configuration~\eqref{config:A}, it is treated as unknown by the proposed likelihood procedure.

\medskip
\noindent\textit{Results.}
Table~\ref{tab:configA} summarizes the results. The estimation errors of the proposed estimators generally decrease with increasing $N$ and $n$, both for complete observations and under MAR dropout. The AIC procedure also tends to select larger truncation levels for $b_1$ and $\sigma$ as $N$ increases, consistent with the increasing sieve dimension suggested by the asymptotic theory.

Under complete observations, the likelihood and benchmark estimators perform similarly for $b_1$, whereas the likelihood estimator yields substantially smaller errors for $\sigma^2$. Under MAR dropout, the advantage of the likelihood approach becomes more pronounced. In particular, the local-polynomial moment estimators exhibit persistent bias because the observed measurements are outcome-dependent. This effect is clearly visible for $\widehat b_{1,\mathrm{MSP}}$ in Figure~\ref{fig:Sim_Results} and is consistent with the discussion in Section~\ref{sec:ConsequencesOfMissingData}. We emphasize that the method of \citet{mohammadi2024nonparametric} was not designed for data subject to MAR dropout.

\begin{table}
 \centering 
\resizebox{\textwidth}{!}{%
 \begin{tabular}{|c|cc|ccc|ccccc|} 
 \hline
 \multicolumn{11}{|c|}{\textsc{Configuration}~\eqref{config:A}} \\
 \hline 
 Missingness & $n$ & $N$ & $M_{b_1}$ & $M_{b_2}$ & 
$M_\sigma$ & $\widehat{b}_{1N}$ & $\widehat{b}_{2N}$ & 
$\widehat{\sigma}^2_N$ & $\widehat{b}_{1,\text{MSP}}$ & 
$\widehat{\sigma}_{\text{MSP}}^2$ \\ 
 \hline 

 & 5  & 100 & 1.67   & 0.19   & 
2.76   & 0.32   & 0.11   & 0.31   & 0.25   & 0.88   
 \\ 

 &    &     & (1.08) & (0.53) & 
(0.76) & (0.28) & (0.31) & (0.22) & (0.10) & (0.28) 
\\ 

               &    & 200 & 2.05   & 0.22   & 
2.94   & 0.19   & 0.09   & 0.20   & 0.21   & 0.75   
 \\ 

               &    &     & (0.85) & (0.54) & 
(0.76) & (0.19) & (0.22) & (0.14) & (0.07) & (0.21) 
\\ 

               &    & 500 & 2.31   & 0.20   & 
3.25   & 0.12   & 0.06   & 0.12   & 0.16   & 0.61   
 \\ 

$\delta_0 = \infty$                 &    &     & (0.81) & (0.53) & 
(0.73) & (0.06) & (0.16) & (0.05) & (0.05) & (0.15) 
\\ 

(\textsc{Complete})                  & 10 & 100 & 2.25   & 0.17   & 
2.82   & 0.15   & 0.10   & 0.16   & 0.21   & 0.77   
 \\ 

               &    &     & (0.70) & (0.45) & 
(0.78) & (0.09) & (0.28) & (0.05) & (0.08) & (0.23) 
\\ 

               &    & 200 & 2.39   & 0.22   & 
3.16   & 0.13   & 0.09   & 0.13   & 0.17   & 0.65   
 \\ 

               &    &     & (0.84) & (0.54) & 
(0.81) & (0.06) & (0.21) & (0.04) & (0.06) & (0.17) 
\\ 

               &    & 500 & 2.54   & 0.20   & 
3.67   & 0.10   & 0.05   & 0.09   & 0.14   & 0.53   
 \\ 

               &    &     & (0.97) & (0.50) & 
(0.86) & (0.03) & (0.13) & (0.02) & (0.04) & (0.13) 
\\ 
 \hline 

     & 5  & 100 & 1.43   & 0.70   & 
2.48   & 0.45   & 0.46   & 0.30   & 0.53   & 0.77   
 \\ 

               &    &     & (1.15) & (1.06) & 
(0.67) & (0.30) & (0.71) & (0.15) & (0.15) & (0.25) 
\\ 

               &    & 200 & 1.95   & 0.71   & 
2.68   & 0.28   & 0.35   & 0.20   & 0.49   & 0.66   
 \\ 

               &    &     & (0.96) & (1.06) & 
(0.75) & (0.23) & (0.56) & (0.09) & (0.10) & (0.17) 
\\ 

               &    & 500 & 2.27   & 0.58   & 
3.08   & 0.17   & 0.17   & 0.14   & 0.46   & 0.55   
 \\ 

$\delta_0 = 3$                     &    &     & (0.85) & (1.02) & 
(0.79) & (0.13) & (0.35) & (0.04) & (0.06) & (0.13) 
\\ 

(\textsc{MAR})                     & 10 & 100 & 1.72   & 0.24   & 
2.75   & 0.35   & 0.17   & 0.22   & 0.60   & 0.71   
 \\ 

               &    &     & (1.08) & (0.64) & 
(0.75) & (0.29) & (0.46) & (0.09) & (0.11) & (0.20) 
\\ 

               &    & 200 & 2.13   & 0.20   & 
2.95   & 0.21   & 0.11   & 0.15   & 0.57   & 0.63   
 \\ 

               &    &     & (0.90) & (0.51) & 
(0.75) & (0.19) & (0.29) & (0.06) & (0.07) & (0.15) 
\\ 

               &    & 500 & 2.41   & 0.22   & 
3.55   & 0.13   & 0.08   & 0.10   & 0.55   & 0.54   
 \\ 

               &    &     & (0.87) & (0.56) & 
(0.84) & (0.07) & (0.20) & (0.02) & (0.05) & (0.11) 
\\ 
 \hline 
 \end{tabular}
}
\caption{Results from $1{,}000$ simulation runs. The first three columns give the average truncation levels selected by the forward AIC procedure, while the remaining columns show the mean RISE for the respective SDE parameter estimators. Standard deviations are reported in parentheses.} 
 \label{tab:configA} 
 \end{table}

\begin{figure}[p]
    \centering
    \caption*{\textsc{Configuration~\eqref{config:A}}}
    \begin{subfigure}[b]{\textwidth}
        \centering
        \caption*{(\textsc{Complete}: $\delta_0 = \infty$)}
    \includegraphics[width=0.9\textwidth]{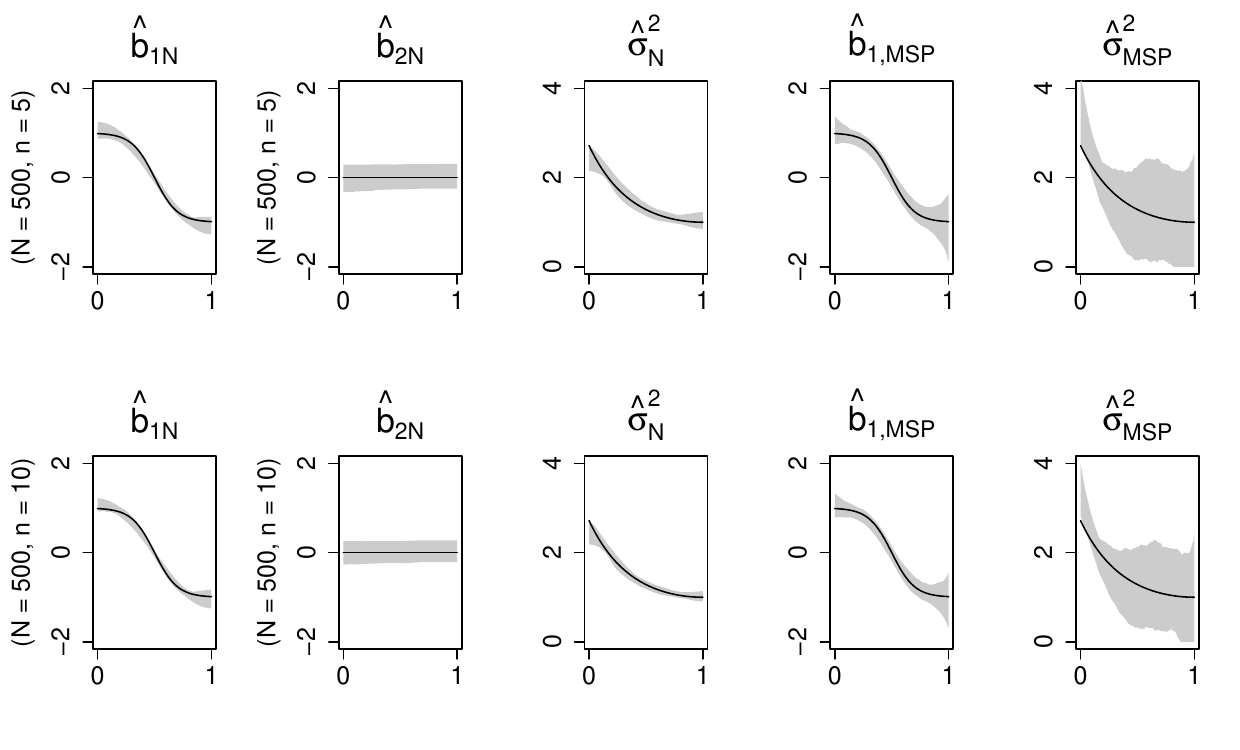}
    \end{subfigure}
    \vspace{1em}
    \begin{subfigure}[b]{\textwidth}
        \centering
        \caption*{(\textsc{MAR}: $\delta_0 = 3$)}
        \includegraphics[width=0.9\textwidth]{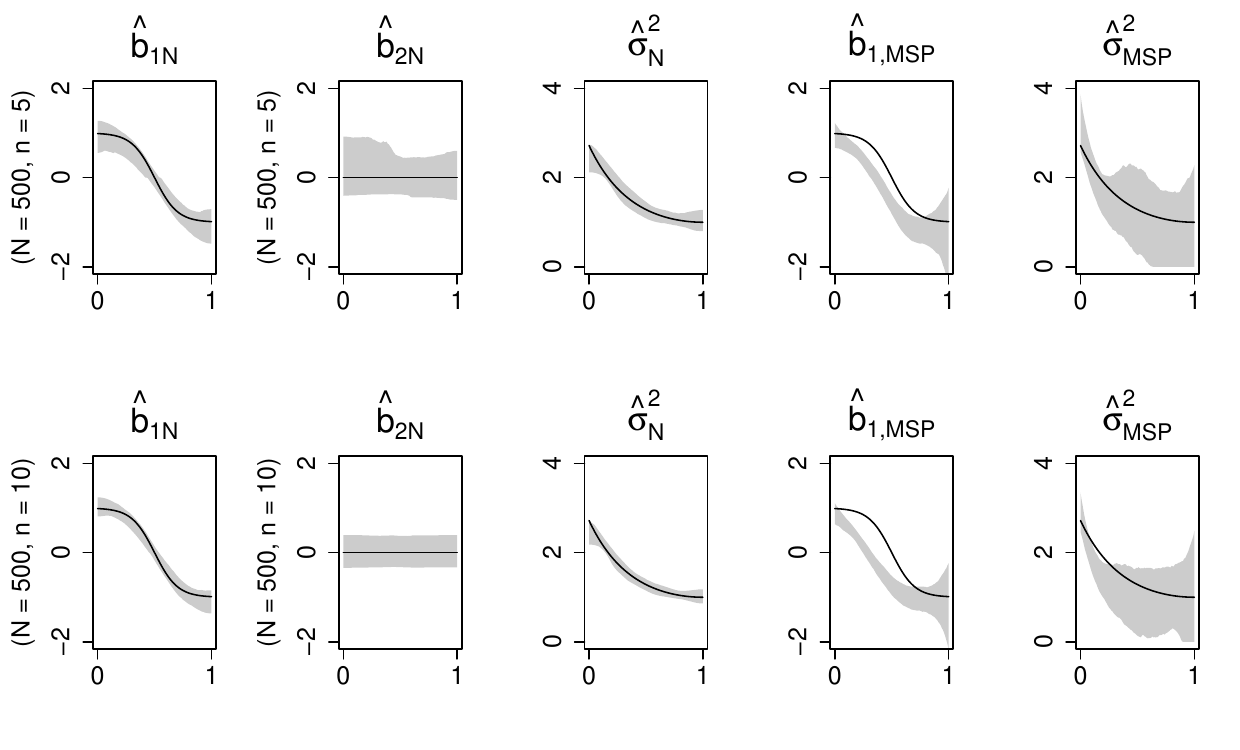}
    \end{subfigure}
    \caption{Results from $1{,}000$ simulation runs with $N=500$ and $n\in\{5,10\}$ measurement times. Gray bands represent 90\% pointwise confidence bands based on the estimated functions. The top row corresponds to complete sparse observations, while the bottom row corresponds to ODD under MAR.}\label{fig:Sim_Results}
\end{figure}

\subsection{Tumor-growth configuration}

\textit{Data-generating process.}
The second configuration is chosen to resemble the log-transformed tumor-growth data analyzed in Section~\ref{sec:RealDataIllustration}. We use mechanism~\eqref{M} with threshold $\delta_0=\log(2)$ and set
\begin{equation}\tag{B}\label{config:B}
    \begin{aligned}
        b_1(t) &= -5.5 \,\varphi_1(t), & \sigma(t) &= \exp(0.3 \,\varphi_1(t) + 0.4 \,\varphi_2(t)), \\
        b_2(t) &= -0.6 \,\varphi_1(t) - 3.5 \,\varphi_2(t), \qquad & (m_0, v_0) &= (-8.8, 2.9),
    \end{aligned}
\end{equation}
where $\varphi_1(t) = 1$ and $\varphi_2(t) = \sqrt{2} \cos(\pi t)$ for $t \in [0,1]$. In contrast to configuration~\eqref{config:A}, the initial value is random.

\medskip
\noindent
\textit{Results.}
Table~\ref{tab:configB} shows that estimation accuracy generally improves as $N$ or $n$ increases. The true functional parameters have finite basis representations with
$
(M_{b_1},M_{b_2},M_\sigma)=(1,2,2).
$
The AIC procedure selects truncation levels close to these values on average and remains stable across the considered sample sizes and sampling frequencies. Overall, the results indicate that the proposed estimator performs well under a data-generating mechanism resembling the tumor-growth application, despite outcome-dependent dropout.

\begin{table}
 \centering 
 \begin{tabular}{|c|cc|ccc|ccccc|} 
 \hline
 \multicolumn{11}{|c|}{\textsc{Configuration}~\eqref{config:B}} \\
 \hline 
 Missingness & $n$ & $N$ & $M_{b_1}$ & $M_{b_2}$ & $M_\sigma$ & $\widehat{b}_{1N}$ & $\widehat{b}_{2N}$ & $\widehat{\sigma}^2_N$ & $\widehat{m}_0$ & $\widehat{v}_0$ \\ 
 \hline 

 & 5  & 100 & 1.15   & 2.14   & 2.20   & 0.37   & 0.65   & 0.21   & 0.19   & 0.44   \\ 
 &    &     & (0.43) & (0.43) & (0.49) & (0.39) & (0.53) & (0.14) & (0.14) & (0.33) \\ 

 &    & 200 & 1.15   & 2.13   & 2.20   & 0.28   & 0.47   & 0.16   & 0.13   & 0.31   \\ 
 &    &     & (0.43) & (0.37) & (0.52) & (0.26) & (0.37) & (0.10) & (0.10) & (0.25) \\ 

 &    & 500 & 1.16   & 2.15   & 2.17   & 0.21   & 0.31   & 0.11   & 0.08   & 0.20   \\ 
$\delta_0 = \log(2)$ &    &     & (0.44) & (0.40) & (0.45) & (0.17) & (0.22) & (0.07) & (0.06) & (0.16) \\ 

(\textsc{MAR}) & 10 & 100 & 1.14   & 2.15   & 2.20   & 0.29   & 0.50   & 0.14   & 0.15   & 0.39   \\ 
 &    &     & (0.40) & (0.43) & (0.50) & (0.28) & (0.39) & (0.08) & (0.11) & (0.30) \\ 

 &    & 200 & 1.14   & 2.18   & 2.14   & 0.23   & 0.38   & 0.11   & 0.11   & 0.27   \\ 
 &    &     & (0.39) & (0.48) & (0.38) & (0.20) & (0.30) & (0.06) & (0.08) & (0.20) \\ 

 &    & 500 & 1.14   & 2.15   & 2.18   & 0.18   & 0.24   & 0.09   & 0.07   & 0.17   \\ 
 &    &     & (0.39) & (0.44) & (0.47) & (0.13) & (0.17) & (0.04) & (0.05) & (0.13) \\ 
 \hline 
 \end{tabular} 
 \caption{Results from $1{,}000$ simulation runs. Reported quantities are defined as in Table~\ref{tab:configA}.}\label{tab:configB}
\end{table}

\section{Tumor growth data}
\label{sec:RealDataIllustration}

We illustrate the proposed Gauss--Markov framework using the tumor-growth data introduced in Section~\ref{sec:intro}. The analysis is based on the SDE parameters estimated by the likelihood procedure of Section~\ref{sec:Estimation}.

\begin{example}[Tumor growth]
Let $Y$ denote the tumor-volume process and suppose that
\[
X(t)=\log Y(t), \qquad t\in[0,T],
\]
satisfies SDE~\eqref{eq:SDE}. By Itô's formula \citep[Theorem~4.1.2]{oksendal2007stochastic},
\begin{equation}\label{eq:SDE_Gompertz}
\diff Y(t)
=
Y(t)
\left(
b_1(t)\log Y(t)+b_2(t)+\frac{\sigma^2(t)}{2}
\right)\diff t
+
\sigma(t)Y(t)\,\diff B(t), \qquad t\in[0,T].
\end{equation}
Thus, on the original scale, the model can be viewed as a time-inhomogeneous stochastic extension of the Gompertz model; see \citet{vaghi2020population} for a homogeneous random-effects formulation. Unlike an ODE-based random-effects model, which generates differentiable trajectories, the SDE formulation allows for non-differentiable sample paths.
\end{example}

Using the forward AIC procedure of Section~\ref{sec:Implementation}, the truncation levels
$
(M_{b_1},M_{b_2},M_\sigma)=(1,2,2)
$
are obtained. The estimated coefficient functions are shown in Figure~\ref{fig:Tumor_Estimates}, with estimated initial parameters
$
(\widehat m_0,\widehat v_0)=(-8.80,2.91).
$
For interpretation, write the drift as
\[
b_1(t)
\left(
X(t)+\frac{b_2(t)}{b_1(t)}
\right).
\]
When $b_1(t)<0$, the process exhibits mean reversion toward the time-varying level $-b_2(t)/b_1(t)$, at a rate determined by $|b_1(t)|$.

In the following, we use the fitted model for three inferential tasks: constructing prediction bands, estimating first-passage times, and predicting tumor age.

\begin{figure}
    \centering
    \includegraphics[width=\linewidth]{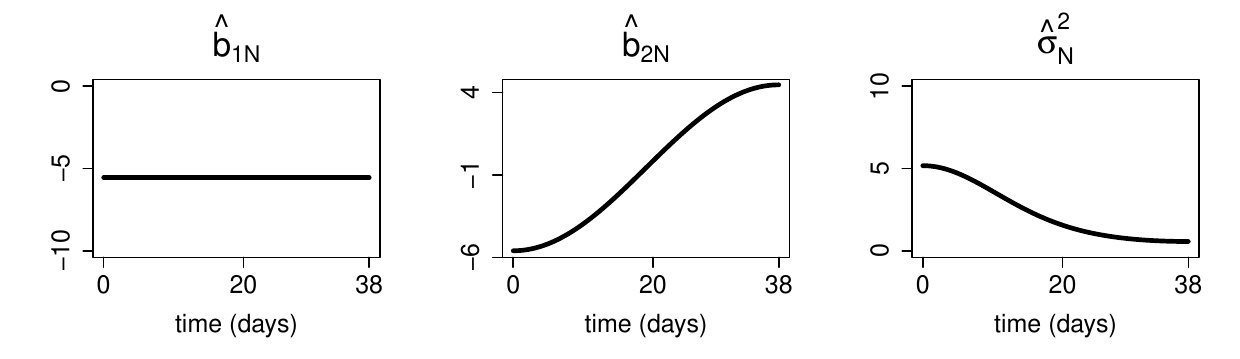}
    \caption{Estimates of the model parameters for the tumor growth data.}
    \label{fig:Tumor_Estimates}
\end{figure}

\subsection{Prediction Bands}

Proposition~\ref{prop:CondMeanVar} directly yields pointwise prediction bands. Let $q(\alpha)$ denote the $\alpha$-quantile of the standard normal distribution.

\begin{proposition}[Pointwise prediction bands]\label{prop:PPB}
The pointwise $\alpha$-quantile of $X$ is
\[
Q_{X(t)}(\alpha)
=
\mu(t)
+
q(\alpha)\Phi(0,t)
\left(
v_0+\int_0^t\Phi(0,u)^{-2}\sigma^2(u)\,\diff u
\right)^{1/2},
\qquad t\in[0,T].
\]
\end{proposition}

Simultaneous prediction bands follow from the Doob representation
\[
X(t)=\mu(t)+h_0(t)B(h_1(t)), \qquad t\in[0,T],
\]
with $h_0$ and $h_1$ given in Proposition~\ref{prop:Doob}. Following \citet{delong1981crossing}, let $c_\alpha$ be chosen so that the corresponding Brownian-motion boundary-crossing probability equals $1-\alpha$. This gives the following simultaneous band.

\begin{proposition}[Simultaneous prediction bands]\label{prop:SPB}
Let $v_0>0$ and $\alpha\in(0,1)$. Then
\[
\prob\left(
|X(t)-\mu(t)|
\le
c_\alpha h_0(t)\sqrt{h_1(t)}
\text{ for all }t\in[0,T]
\right)
=
1-\alpha.
\]
\end{proposition}

The construction is computationally inexpensive and requires no simulation. Figure~\ref{fig:tumor_bands_fpt} (left) shows the estimated pointwise median together with 75\% pointwise and simultaneous prediction bands.

\begin{figure}
    \centering
    \includegraphics[width=0.95\linewidth]{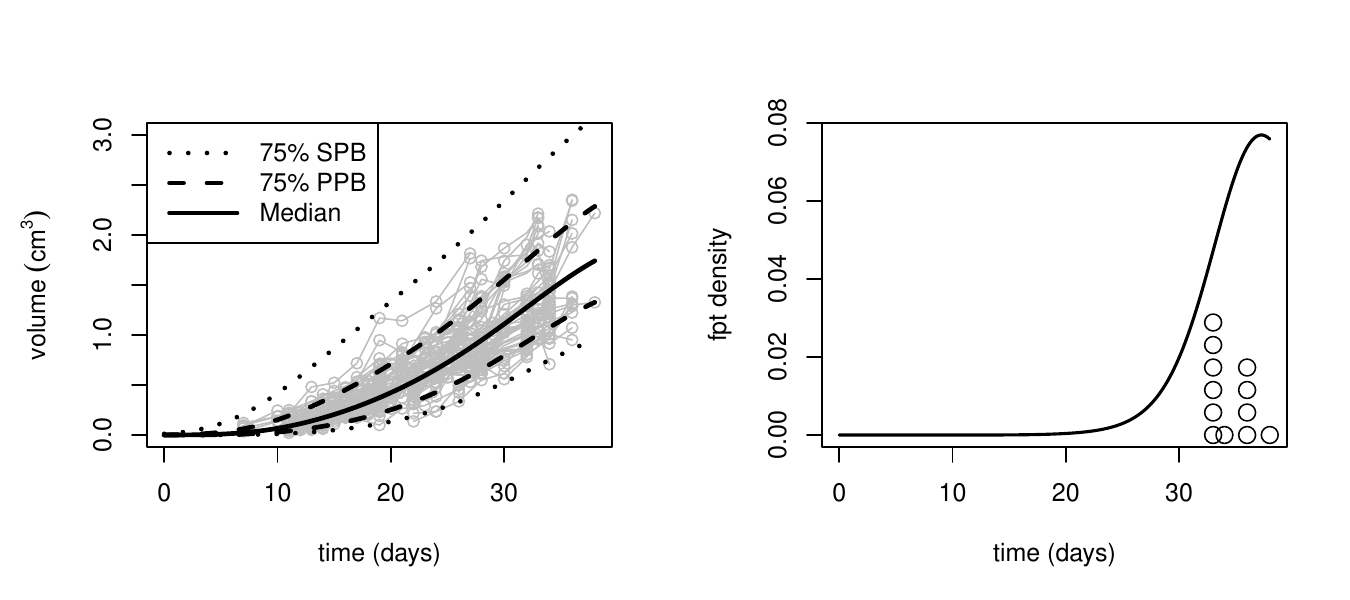}
    \caption{Left: Tumor growth trajectories for $N = 66$ mice, together with the pointwise median (solid), 75\% pointwise prediction band (dashed), and 75\% simultaneous prediction band (dotted). Right: Estimated first-passage time density $\widehat{g}_\delta$ for the threshold $\delta = 2\,\text{cm}^3$; points indicate observed threshold exceedance times.}
    \label{fig:tumor_bands_fpt}
\end{figure}

\subsection{First-passage time}

For a threshold $\delta$ and initial time $t_0<T$ such that $X(t_0)=x_0<\delta$, define
\[
\tau_\delta = \inf\{t\in[t_0,T]\colon X(t)=\delta\},
\]
with $\tau_\delta=\infty$ if the threshold is not reached by time $T$. Its (possibly defective) density is
\[
g_\delta(t\vert x_0,t_0)
=\frac{\diff}{\diff t}
\prob(\tau_\delta\le t\vert X(t_0)=x_0).
\]

The density satisfies the Volterra equation of the second kind
\begin{equation}\label{eq:VolterraIntegralEquation}
g_\delta(t\vert x_0,t_0)
=
-2H(\delta,t\vert x_0,t_0)
+
2\int_{t_0}^t
g_\delta(s\vert x_0,t_0)
H(\delta,t\vert\delta,s)\,\diff s,
\end{equation}
where
$H(x_1,t\vert x_0,t_0)=\frac{\diff}{\diff t}
\prob(X(t)\le x_1\vert X(t_0)=x_0)$;
see \citet{buonocore1987integral,gutierrez1997first}. We approximate $g_\delta$ using the iterative procedure of \citet{nardo2001computational}.
For the tumor data, we set $\delta=2\,\mathrm{cm}^3$ and $(x_0,t_0)=(\widehat m_0,0)$. The estimated density is shown in the right panel of Figure~\ref{fig:tumor_bands_fpt}. The estimated probability of reaching the threshold within the study period is
\[
\widehat{\prob}(\tau_\delta<T\vert X(0)=x_0)=\int_{t_0}^T\widehat{g}_\delta(t\vert x_0,t_0)\,dt=0.49,
\]
and the estimated conditional mean passage time, given that the threshold is reached, is 
\[
\widehat{\expv}[\tau_\delta\vert \tau_\delta < T, X(t_0) = x_0] = \int_{t_0}^T t \, \widehat{g}_\delta(t\vert x_0, t_0) \, \diff t 
=34.01\text{ days}.
\]

\subsection{Tumor age prediction}

The fitted nonstationary model can also be used to estimate an unknown temporal shift. Suppose that
\[
W(t)=X(t+\lambda), \qquad t\in[0,T-\lambda],
\]
is observed at $0\le t_1<\cdots<t_n\le T-\lambda$, with $w_j=x(t_j+\lambda)$. We estimate the shift $\lambda$ by
\begin{equation}\label{eq:TimeWarpingDensity}
\widehat\lambda
=
\argmax_{\lambda\in[0,T-t_n]}
f_{X(t_1+\lambda),\ldots,X(t_n+\lambda)}
(w_1,\ldots,w_n).
\end{equation}

We apply this procedure to tumor-age prediction; see \citet{vaghi2020population} for an alternative approach. The model is fitted to a training sample of $N_{\mathrm{train}}=44$ mice. Of the remaining subjects, five are excluded because they have no measurements beyond three weeks, leaving $N_{\mathrm{test}}=17$. For each test subject, measurements from the first three weeks after tumor injection are withheld, and $\widehat\lambda$ is used to estimate tumor age from the remaining observations. Figure~\ref{fig:tumor_age} summarizes the resulting prediction errors.

\begin{figure}
    \centering
    \includegraphics[width=\linewidth]{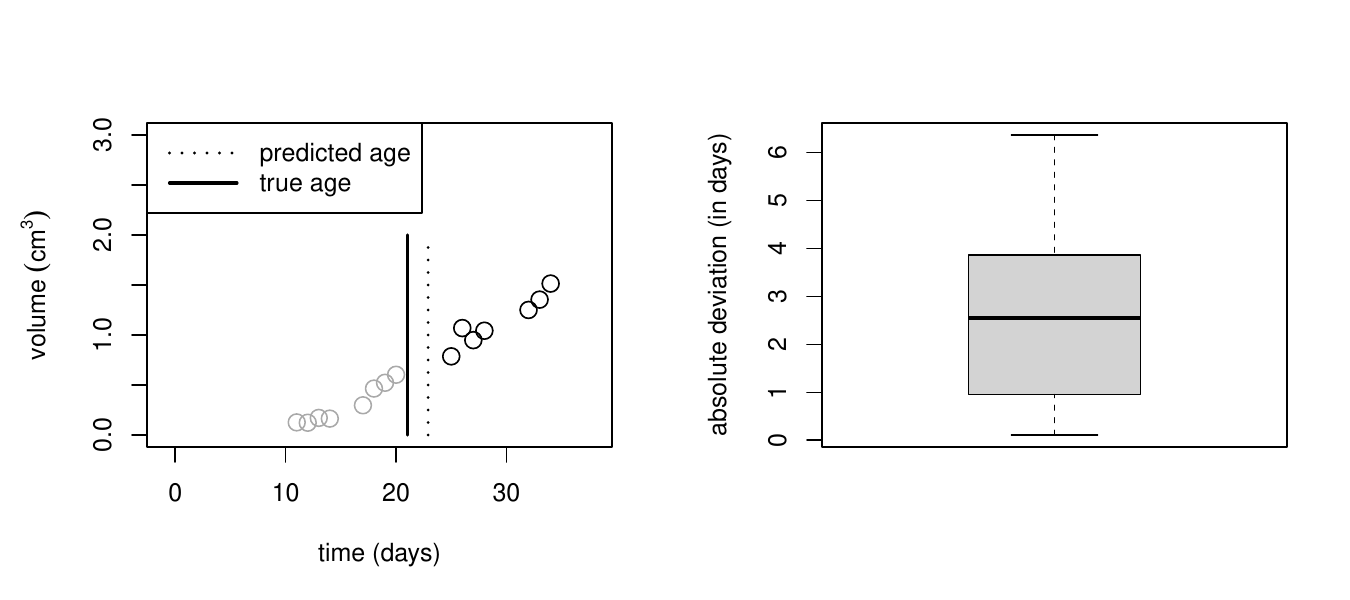}
    \caption{Left: Tumor volume measurements for a representative subject, with withheld observations shown in gray and the estimated tumor age of $22.90$ days indicated by the dotted vertical line. Right: Distribution of absolute tumor-age prediction errors across the test subjects.}
    \label{fig:tumor_age}
\end{figure}

\section{Conclusion}\label{sec:Conclusions}

We developed a likelihood-based framework for distributional modeling of functional data using Gauss--Markov processes generated by linear stochastic differential equations. The proposed approach directly estimates the drift and diffusion functions through sieve maximum likelihood and accommodates sparse, irregularly observed trajectories. An important feature is its ability to handle outcome-dependent dropout: although such sampling violates the commonly imposed MCAR assumption, we showed that the resulting missingness mechanism is MAR and can therefore be accommodated by likelihood-based inference.

We established convergence rates for the proposed estimators and showed that the drift and diffusion coefficients exhibit different degrees of statistical difficulty, with the diffusion coefficient admitting a faster rate. Simulation experiments supported the theoretical findings and demonstrated the advantages of likelihood-based estimation over moment-based methods in the presence of outcome-dependent dropout. Finally, the analysis of tumor-growth data illustrated how the fitted distributional model can be used for inference beyond the estimation of mean and covariance functions, including the construction of prediction bands, first-passage time analysis, and tumor-age prediction.

The proposed framework also provides a natural approach to reconstruct partially observed functional data. Related methods include the ridge-regularized functional completion approach of \citet{kraus2015components}, the linear reconstruction methodology of \citet{kneip2020on}, and the discrete Markov-chain approach of \citet{delaigle2016approximating}. In our setting, the Gauss--Markov structure directly yields the conditional distribution of an unobserved segment given the available measurements, thereby providing both a reconstruction of the missing part of a trajectory and a natural quantification of its uncertainty.

Several other extensions are of interest but beyond the scope of the present paper. These include more general nonlinear or non-Gaussian SDE models, alternative observation and dropout mechanisms, and a more detailed investigation of the optimal convergence rates for drift and diffusion estimation under sparse sampling.

\section*{Funding}

 Alexander Aue's research was partially funded by National Science Foundation grant DMS-2515821. Siegfried H\"or\-mann's research was funded in whole or in part by the Austrian Science Fund (FWF) [\href{https://doi.org/10.55776/P35520}{10.55776/P35520}]. For open access purposes, the authors have applied a CC BY public copyright license to any author-accepted manuscript version arising from this submission. Maximilian Ofner gratefully acknowledges support from the Austrian Marshall Plan Foundation.

 \section*{Acknowledgments}

The computational results were obtained using the Austrian Scientific Computing (ASC) infrastructure. The authors thank Victor Panaretos and Leonardo Santoro for sharing implementation details of the method proposed by \cite{mohammadi2024nonparametric}.

\putbib[ref]
\end{bibunit}

\begin{bibunit}


\clearpage

\appendix

\section{Moment and likelihood estimation under MAR}
\label{app:A}

The examples below demonstrate a MAR setting that yields inconsistent moment estimates, but consistent likelihood estimates.

\begin{example}[Moment estimation under MAR]\label{ex:MARMoment}
To illustrate the effect of outcome-dependent dropout, consider the simplest possible setting with two fixed observation times $0\le s<t\le T$.
The information contained in the process $X$ reduces to the bivariate Gaussian vector
\[
(X_1,X_2)=(X(s),X(t)).
\]
Let $\{X_2\textup{ obs}\}$ denote the event that the second observation is available. If the second observation is missing, only $X_1$ is observed. Proposition~\ref{prop:MAR} therefore reduces to the familiar MAR condition
\[
\prob(X_2\textup{ obs}\vert X_1,X_2)
=
\prob(X_2\textup{ obs}\vert X_1).
\]
Suppose the objective is to estimate $\mu_2=\expv[X_2]$. A natural estimator is the sample mean computed from the available observations,
\[
\widehat\mu_2
=
\frac{\sum_{i=1}^N X_{i2} \ind\{\textup{$X_{i2}$ obs}\}}{\sum_{i=1}^N \ind\{\textup{$X_{i2}$ obs}\}}.
\]
By the law of large numbers,
\[
\widehat\mu_2
\;\xrightarrow{p}\;
\expv[X_2\vert X_2\textup{ obs}],
\]
which generally differs from the target quantity $\mu_2$. Thus, under MAR, the naive moment estimator is generally inconsistent.
\end{example}

\begin{example}[Likelihood estimation under MAR]\label{ex:MARLikelihood}
Consider the setting of Example~\ref{ex:MARMoment}. By \eqref{eq:SolutionLSDE},
\[
X_{i2}=\beta_0+\beta_1X_{i1}+\varepsilon_i,
\qquad i=1,\ldots,N,
\]
where $\beta_0 =\int_s^t\Phi(u,t)b_2(u)\,du$, $\beta_1=\Phi(s,t)$, and $\varepsilon_i=\int_s^t\Phi(u,t)\sigma(u)\,dB_i(u)$. The errors are independent of $X_{i1}$ and satisfy
\[
\expv[\varepsilon_i\vert X_{i1}]=0.
\]
Suppose that $(X_{i1},X_{i2})$ is observed only when the second measurement is available. \cite{anderson1957maximum} showed that the maximum likelihood estimator of $\mu_2=\expv[X_2]$ is
\[
\widehat\mu_2 =
\overline X_{2,\mathrm{obs}}
+
\widehat{\beta}_1\left(
\overline X_1-
\overline X_{1,\mathrm{obs}}
\right),
\]
where
\begin{align*}
\overline X_{1,\mathrm{obs}}
=
\frac{\sum_{i=1}^N X_{i1}\ind\{X_{i2}\textup{ obs}\}}
{\sum_{i=1}^N\ind\{X_{i2}\textup{ obs}\}}, && \overline X_{2,\mathrm{obs}}
=
\frac{\sum_{i=1}^N X_{i2}\ind\{X_{i2}\textup{ obs}\}}
{\sum_{i=1}^N\ind\{X_{i2}\textup{ obs}\}},
\end{align*}
and $\widehat{\beta}_1$ is the slope estimator computed on all available pairs. Under MAR, 
\[\prob(X_2\textup{ obs}\vert X_1,X_2)=\prob(X_2\textup{ obs}\vert X_1),\]
and the observation indicator is thus conditionally independent of the regression error given $X_1$. Consequently,
$
\cov(\varepsilon,X_1\vert X_2\textup{ obs})=0,
$\
so that
\[
\frac{\cov(X_1,X_2\mid X_2\textup{ obs})}{\var(X_1\mid X_2\textup{ obs})}=\beta_1.
\]
It follows that
\[
\widehat\mu_2
\;\xrightarrow{p}\;
\expv[X_2],
\]
showing that, unlike the naive moment estimator, the likelihood estimator remains consistent under MAR.
\end{example}

\section{Proofs of the main results}
\label{sec:MainProofs}

This section provides the proofs for the results presented in the main text. Auxiliary results are detailed in Sections~\ref{sec:ProofOfLemma}--\ref{sec:Norms}. We adopt the following notation throughout.

\textit{Notation.}
Consider functions $a$ and $b$ defined on some arbitrary domain. We write $a \lesssim b$, if there exists a universal constant $C > 0$ with $a < Cb$ for every point-wise evaluation of the functions. We write $a \simeq b$, if both $a \lesssim b$ and $b \lesssim a$. Constants $C_1, C_2, \dots$ may vary between proofs, whereas the constants $c_0, c_1, \dots$ remain fixed throughout the Appendix.

\begin{proof}[Proof of Proposition~\ref{prop:CondMeanVar}]
Let $f\colon [0,T] \to \mathbb{R}$ be a deterministic and square-integrable function. Recall the following basic properties of Itô integrals:
\begin{itemize}
    \item $\int_s^tf(u) \, \diff B(u)$ is Gaussian,
    \item $\expv[\int_s^t f(u) \, \diff B(u)] = 0$,
    \item $\var\left(\int_s^t f(u) \, \diff B(u)\right) = \int_s^t f^2(u) \, \diff u$,
\end{itemize}
for any $0 \leq s \leq t \leq T$. Equation~\eqref{eq:SolutionLSDE} states that
\begin{equation*}
X(t) = \Phi(s,t)X(s) + \int_s^t \Phi(u,t)b_2(u)\diff u + \int_s^t \Phi(u,t)\sigma(u)\diff B(u).
\end{equation*}
A property of the Itô integral implies $\expv[\int_s^t \Phi(u,t)\sigma(u)\,\diff B(u)] = 0$. Using independence of $X(s)$ and $(B(u)\colon s \leq u \leq t)$, we obtain the expression for the conditional mean,
\begin{equation*}
\expv[X(t)\vert X(s)] = \Phi(s,t)X(s) +  \int_s^t \Phi(u,t)b_2(u)\diff u.
\end{equation*}
Concerning the conditional variance, the Itô isometry \cite[Corollary~3.1.7]{oksendal2007stochastic} yields
\begin{equation*}
\var(X(t) \vert X(s)) = \expv\left[ \left( \int_s^t \Phi(u,t)\sigma(u)\diff B(u)\right)^2\right]=\int_s^t \Phi^2(u,t) \sigma^2(u)\, \diff u.
\end{equation*}
Similar arguments can be used to derive the other expressions for the unconditional moments.
\end{proof}

\begin{proof}[Proof of Proposition~\ref{prop:Doob}]
    The distribution of a Gaussian process is uniquely determined by its mean and covariance function. So we need to choose $h_0$ and $h_1$ such that the covariance
    \begin{equation*}
        \cov\left(h_0(s)B(h_1(s)), h_0(t)B(h_1(t))\right) = h_0(s) h_0(t) \min\{h_1(s), h_1(t)\},
    \end{equation*}
    matches the one of $\cov(X(s),X(t))$ for $0 \leq s \leq t \leq T$. Proposition~\ref{prop:CondMeanVar} shows that $\cov(X(s),X(t))$ is triangular in the sense that
    \begin{equation*}
        \cov(X(s),X(t)) = \Phi(0,s) \Phi(0,t) \left(v_0 + \int_0^s \Phi(0,u)^{-2}\sigma^2(u) \, \diff u\right) .
    \end{equation*}
    Setting $h_0(t) = \Phi(0,t)$ and $h_1(t) = v_0 + \int_0^t \Phi(0,u)^{-2} \sigma^2(u) \, \diff u$ then yields the desired form because~$h_1$ is an increasing function.
\end{proof}

\begin{proof}[Proof of Proposition~\ref{prop:MAR}]
    Let $(\Omega, \mathcal{A}, \mathbb{P})$ be the underlying probability space and $K \subset [0,T]$ be arbitrary but fixed. Define a new random variable $N^* = \min \{i\colon X(T_i) > \delta_0\}$
    with the convention that $N^* = n$ in case $X(T_i) \leq \delta_0$ for all $i = 1, \dots, n$. The indicator $\ind\{O \subset K\}$ can then be written as the function
    \begin{equation*}
        f_K(T_1, \dots, T_n, X) = \sum_{\ell=1}^n \ind\{T_1, \dots, T_\ell \in K\} \ind \{N^* = \ell\}.
    \end{equation*}
     Because the random locations $T_1, \dots, T_n$ are independent of $X$ by assumption, the conditional probability satisfies
    \begin{equation*}
        \prob(O \subset K\vert \mathcal{F}_{[0,T]})(\omega) = \expv \left[f_K(T_1, \dots, T_n, X(\omega))\right], \qquad \omega \in \Omega.
    \end{equation*}
    Notice that the event $\{T_1, \dots, T_\ell \in K\} \cap \{N^*= \ell\}$ depends on the path of $X$ only through its values on~$K$.
    Therefore,
    \[\omega \mapsto \expv \left[f_K(T_1, \dots, T_n, X(\omega))\right]\]
    is $\mathcal{F}_K$-measurable and we conclude that
    \begin{equation*}
        \prob(O \subset K\vert \mathcal{F}_{[0,T]}) = \prob(O \subset K\vert \mathcal{F}_K)
    \end{equation*}
    almost surely.
\end{proof}

\begin{proof}[Proof of Theorem~\ref{thm:ConvergenceRate}]
We only derive the convergence rate for $\widehat{\sigma}_{N}$, the other results follow analogously. By the triangle inequality and the fact $\normwww{f} \leq \norm{f}_{L^2}$, we obtain
\begin{align*}
    \normwww{\widehat{\sigma}_{N}^2 - \sigma_{0}^2}^2 \lesssim \normwww{\widehat{\sigma}_{N}^2 - \sigma_{0N}^2}^2 + \norm{\sigma_{0N}^2 - \sigma_0^2}_{L^2}^2.
\end{align*}
Lemma~\ref{lem:Conv} of Section~\ref{sec:ProofOfLemma} shows that
\begin{equation*}
    \normwww{\widehat{\sigma}_{N}^2 - \sigma_{0N}^2}^2 = O_p(\log(M) M N^{-1}).
\end{equation*}
Furthermore, the mean value theorem yields $\norm{\sigma_{0N}^2 - \sigma_0^2}_{L^2}^2 \lesssim \norm{\log \sigma_{0N} - \log \sigma_0}_{L^2}^2$. Therefore, Parseval's identity and Assumption~\ref{ass:beta} imply
\begin{equation*}
    \norm{\sigma_{0N}^2 - \sigma_0^2}_{L^2}^2 \lesssim \norm{\log \sigma_{0N} - \log \sigma_0}_{L^2}^2 = \sum_{k=M+1}^\infty \sigma_{0k}^2 \lesssim \sum_{k=M+1}^\infty k^{-2\beta} \lesssim M^{1-2\beta}.
\end{equation*}
Setting $M \sim N^{1/2\beta}$, we get
\begin{equation*}
      \normwww{\widehat{\sigma}_{N}^2 - \sigma_{0}^2}^2 = O_p(\log(M) M N^{-1}) + O(M^{1-2\beta}) = O_p\left(\log(N) N^{(1-2\beta)/2\beta}\right),
\end{equation*}
which finishes the proof of the theorem.
\end{proof}

\begin{proof}[Proof of Corollary~\ref{cor:L2Rates}]
    The proof follows from a straightforward adaptation of the preceding one, using the reversed inequalities presented in Section~\ref{sec:Norms}. Specifically, for $\widehat{\sigma}_N$, applying the triangle inequality and Lemma~\ref{lem:LowerBoundNormII} yields
    \begin{align*}
    \norm{\widehat{\sigma}_{N}^2 - \sigma_{0}^2}^2_{L^2} \lesssim \norm{\widehat{\sigma}_{N}^2 - \sigma_{0N}^2}^2_{L^2} + \norm{\sigma_{0N}^2 - \sigma_0^2}_{L^2}^2 \lesssim M \normwww{\widehat{\sigma}_{N}^2 - \sigma_{0N}^2}^2 + \norm{\sigma_{0N}^2 - \sigma_0^2}_{L^2}^2.
\end{align*}
    The remainder of the proof proceeds analogously.
\end{proof}

\begin{proof}[Proof of Theorem~\ref{thm:L2}]
    We derive a lower bound for the convergence rate of $\widehat{\sigma}_N$. Rates for $\widehat b_{1N}$ and~$\widehat b_{2N}$ can be derived along similar lines. The proof is based on Theorem~2.5 of \citet{tsybakov2009introduction}.

    Let $P \in \mathcal{P}$ be defined by setting $b_1 = b_2 = 0$ and leaving $\sigma^2$ as a free parameter,
    \begin{align*}
        \sigma(u) = \exp(f(u)), && f(u) = \sum_{k=1}^\infty \sigma_k \varphi_k(u), && u \in [0,1].    \end{align*}
    Let $P_0$ denote the true model corresponding to $\sigma_0 = 1$ and $f_0 = 0$. The Kullback--Leibler divergence satisfies
    \begin{align*}
        \frac{\text{KL}(P, P_0)}{N} \lesssim \int_0^1  \left(\int_0^t \Delta \sigma^2(u)\,\diff u\right)^2 \, \diff t+ \int_0^1\int_0^{t_2} \frac{\left(\int_{t_1}^{t_2} \Delta \sigma^2(u)\,\diff u \right)^2}{(t_2-t_1)^2} \, \diff t_1\, \diff t_2.
    \end{align*}
    where $\Delta \sigma^2(u) = \sigma^2(u) - 1$, for $u \in [0,1]$. A Taylor expansion yields $\sigma^2(u) - 1 = 2f(u) + r(u)f(u)^2$ where $\norm{r}_\infty < \infty$. Therefore,
        \begin{align*}
        \frac{\text{KL}(P, P_0)}{N} \lesssim \int_0^1 \left(\int_0^t f(u)\,\diff u\right)^2 \, \diff t+ \int_0^1\int_0^{t_2} \frac{\left(\int_{t_1}^{t_2} f(u)\,\diff u \right)^2}{(t_2-t_1)^2} \, \diff t_1\, \diff t_2 + R(f),
    \end{align*}
    where the remainder term is given by
    \begin{equation*}
        R(f) \lesssim \int_0^1 \left(\int_0^t f(u)^2\,\diff u\right)^2 \, \diff t + \int_0^1\int_0^{t_2} \frac{\left(\int_{t_1}^{t_2} f(u)^2\,\diff u \right)^2}{(t_2-t_1)^2} \, \diff t_1\, \diff t_2.
    \end{equation*}
    To simplify notation, write
    \begin{equation*}
    \langle f, g \rangle = \int_0^1 \left(\int_0^t f(u) \diff u\right)\left(\int_0^t g(u) \diff u\right) \, \diff t + \int_0^1 \int_{0}^{t_2} \frac{\left(\int_{t_1}^{t_2} f(u) \diff u\right)\left(\int_{t_1}^{t_2} g(u) \diff u\right)}{(t_2-t_1)^2} \, \diff t_1 \, \diff t_2,
    \end{equation*}
    and observe that
    \begin{equation}\label{eq:KLsigma}
        \text{KL}(P, P_0) \leq C_1 N (\langle  f,f\rangle + R(f)),
    \end{equation}
    for some constant $C_1 > 0$.

    Set $m = \lceil c_0 N^{1/(2\beta + 1)}\rceil$ and let $\Omega \subset \lbrace\omega = (\omega_1, \dots, \omega_m), \,\omega_i \in \{-1,1\} \rbrace = \{-1, 1\}^m $ be chosen below. For each $\omega \in \Omega$, define
    \begin{equation*}
    f^{(\omega)}(t) = \rho \sum_{k=m+1}^{2m} \omega_{k - m} k^{-\beta}\varphi_k(t), \qquad t \in [0,1],
    \end{equation*}
    where $\rho \in (0,C)$ is a constant which ensures that $\abs{\sigma_k^{(\omega)}} = \rho k^{-\beta}<Ck^{-\beta}$. Then,
    \begin{equation*}
        \langle f^{(\omega)}, f^{(\omega)}\rangle = \rho^2 \sum_{k,\ell=m+1}^{2m} \omega_{k-m} \omega_{\ell-m} k^{-\beta} \ell^{-\beta} \langle \varphi_k,\varphi_\ell\rangle.
    \end{equation*}
    Take $\omega \in \{-1,1\}^m$ to be uniformly distributed and observe
    \begin{equation*}
        \expv[\langle f^{(\omega)}, f^{(\omega)}\rangle] = \rho^2 \sum_{k=m+1}^{2m} k^{-2\beta} \langle \varphi_k,  \varphi_k \rangle.
    \end{equation*}
    Because there exists a constant $C_2>0$ such that $\langle \varphi_k,\varphi_k\rangle < C_2 k^{-1}$ for all $k\geq 1$, it follows that
    \begin{equation}\label{eq:MarkovP1}
        \expv[\langle f^{(\omega)}, f^{(\omega)}\rangle]  = \rho^2 \sum_{k=m+1}^{2m} k^{-2\beta} \langle \varphi_k,  \varphi_k \rangle \leq C_2 \sum_{k=m+1}^{2m} k^{-2\beta - 1} \leq C_3 m^{-2\beta},
    \end{equation}
    for some constant $C_3 > 0$. Moreover, Lemma~\ref{lem:UpperBoundNorm} implies
    \begin{equation*}
        R(f^{(\omega)}) \lesssim \int_0^1 (f^{(\omega)}(u))^4 \, \diff u.
    \end{equation*}
    Thus,
    \begin{equation*}
        \expv[R(f^{(\omega)})] \lesssim \sum_{k_1, k_2, k_3, k_4=m+1}^{2m} (k_1 k_2k_3k_4)^{-\beta} \expv[\omega_{k_1-m}\omega_{k_2-m}\omega_{k_3-m} \omega_{k_4-m}] \, \psi(k_1, k_2, k_3,k_4),
    \end{equation*}
    where $\psi(k_1, k_2, k_3, k_4) = \int_0^1 \varphi_{k_1}(u)\varphi_{k_2}(u)\varphi_{k_3}(u)\varphi_{k_4}(u) \, \diff u$. Since $\psi$ is bounded and $\expv[\omega_{k-m}] = 0$ for any $k$, the independence between different entries of $\omega$ implies
    \begin{equation}\label{eq:MarkovP2}
        \expv[R(f^{(\omega)})] \lesssim \sum_{k_1, k_2 = m+1}^{2m} (k_1k_2)^{-2\beta} \leq C_4 m^{-4\beta + 2} \leq C_4 m^{-2\beta},
    \end{equation}
    for a constant $C_4 > 0$ where we have used the assumption that $\beta > 1$. Let $C_5 = 4\max\{C_3, C_4\}$. Using~\eqref{eq:MarkovP1} and~\eqref{eq:MarkovP2}, the Markov inequality implies,
    \begin{equation*}
        \prob\left(\omega \in \{-1,1\}^m\colon \langle f^{(\omega)}, f^{(\omega)}\rangle>C_5 m^{-2\beta} \text{ or } R(f^{(\omega)}) > C_5 m^{-2\beta}\right) \leq 1/2.
    \end{equation*}
        Consequently, there exists a subset $\Omega \subset \{-1,1\}^m$ of cardinality $|\Omega| \geq 2^{m-1}$, such that
    \begin{equation*}
        \langle f^{(\omega)}, f^{(\omega)}\rangle\leq C_5m^{-2\beta}, \qquad R(f^{(\omega)}) \leq C_5 m^{-2\beta}, \qquad \forall \, \omega \in \Omega.
    \end{equation*}
    Let $\rho(\omega, \omega') = \sum_{k=1}^m \mathbbm{1}\{\omega_k \neq \omega_k'\}$ denote the Hamming distance. An adaptation of the Varsha\-mov--Gilbert bound \citep[Lemma~2.9]{tsybakov2009introduction} yields the existence of an integer $M \geq 2^{m/8}$ as well as a subset $\{\omega^{(1)}, \dots, \omega^{(M)}\}$ of cardinality $M$ such that
    \begin{enumerate}[label=(\alph*)]
        \item $\rho(\omega, \omega') \geq \frac{m}{8}$,
        \item $\langle f^{(\omega)}, f^{(\omega)}\rangle\leq C_5m^{-2\beta}$ and $R(f^{(\omega)}) \leq C_5 m^{-2\beta}$,
    \end{enumerate}
    for all $\omega \neq \omega' \in \{\omega^{(1)}, \dots, \omega^{(M)}\}$. This motivates us to define the class of functions $f^{(0)} = 0$ and
    \begin{equation*}
        f^{(j)} = f^{(\omega)}\mid (\omega = \omega^{(j)}), \qquad j = 1, \dots, M.
    \end{equation*}
    The assertion then follows from Theorem~2.5 in \cite{tsybakov2009introduction} once we have established the following two conditions:
\begin{enumerate}[label=(\roman*)]
    \item $\norm{(\sigma^{(i)})^2 - (\sigma^{(j)})^2}_{L^2}^2 \geq 2s$, for all $0 \leq i < j \leq M$,
    \item $\frac{1}{M} \sum_{j=1}^M \text{KL}(P^{(j)}, P^{(0)}) \leq \alpha \log(M)$, for some $\alpha \in (0,1/8)$,
\end{enumerate}
for $s \simeq N^{-(2\beta-1)/(2\beta +1)}$.
Concerning (i),
\begin{equation*}
    \norm{(\sigma^{(i)})^2 - (\sigma^{(j)})^2}_{L^2}^2 \gtrsim \sum_{k=m+1}^{2m} (\omega_{k-m}^{(i)} - \omega_{k-m}^{(j)})^2 k^{-2\beta} \geq \rho(\omega^{(i)}, \omega^{(j)}) (2m)^{-2\beta} \gtrsim m^{-2\beta + 1},
\end{equation*}
where we have used the mean value theorem, Parseval's identity, and property (a).

Concerning (ii), \eqref{eq:KLsigma} and property (b) imply
\begin{equation*}
    \text{KL}(P^{(j)},P^{(0)}) \leq 
    2C_1 C_5 N  m^{-2\beta}.
\end{equation*}
Set $C_6 = 2C_1C_5$. Regarding the constant $c_0$ in the definition of $m$, fix $\alpha \in (0,1/8)$ and choose
\begin{equation*}
    c_0 > \left(\frac{8C_6}{\alpha \log(2) }\right)^{1/(2\beta + 1)}.
\end{equation*}
For large enough $m$ such that $M\geq 2^{m/8}$,
\begin{align*}
    \text{KL}(P^{(j)},P^{(0)}) \leq N C_6 m^{-2\beta} \leq \frac{8 C_6}{\alpha \log(2)}  c_0^{-2\beta -1}  \frac{\alpha \log(2)m}{8} < \alpha \log(M),
\end{align*}
which establishes (ii). Theorem~2.5 in \citet{tsybakov2009introduction} finishes the proof.
\end{proof}

\begin{remark}
    The proofs for $b_1$ and $b_2$ proceed similarly. In these cases, the Kullback--Leibler divergence induces a different inner product that satisfies $\langle\varphi_k,\varphi_k\rangle \lesssim k^{-2}\log(k)$. Consequently, the resulting lower bound yields a slower rate.
\end{remark}

\begin{proof}[Proof of Proposition~\ref{prop:PPB}]
    The proof immediately follows from Proposition~\ref{prop:CondMeanVar}.
\end{proof}

\begin{proof}[Proof of Proposition~\ref{prop:SPB}]
Set $[0,T] = [0,1]$. Doob's representation~\eqref{eq:Doob} and the scaling property of the Brownian motion yield,
\begin{equation*}
\abs{X(t) - \mu(t)}  \overset{d}{=} h_0(t) \abs{B(h_1(t))} \overset{d}{=} h_0(t) \, \sqrt{h_1(1)} \, \abs*{B\left(\frac{h_1(t)}{h_1(1)}\right)}, \qquad t \in [0,1].
\end{equation*}
Furthermore,
\begin{align*}
\abs*{B\left(\frac{h_1(t)}{h_1(1)}\right)} \leq c \sqrt{\frac{h_1(t)}{h_1(1)}} \quad \forall t \in [0,1] && \iff && \abs{B(s)} \leq c \sqrt{s}\quad \forall s \in \left[\varepsilon, 1\right],
\end{align*}
where $\varepsilon = h_1(0)/h_1(1)=v_0/h_1(1)>0$. We then can use the method of \cite{delong1981crossing} to find a threshold $c = c_\alpha$ such that
\begin{equation*}
    \prob\left(\abs{B(s)} \leq c \sqrt{s}, \forall s \in [\varepsilon, 1]\right) = 1-\alpha.
\end{equation*}
This yields
\begin{equation*}
    \prob\left(\abs{X(t)-\mu(t)} \leq c h_0(t) \sqrt{h_1(t)}, \forall t \in [0, 1]\right) = 1-\alpha,
\end{equation*}
and finishes the proof.
\end{proof}

\section{Proof of Lemma~\ref{lem:Conv}}\label{sec:ProofOfLemma}

Our proof strategy for Lemma~\ref{lem:Conv} is based on the theory of $M$-estimators. We heavily rely on the results and notation presented in the book \citep{van2023weak}. Recall that our estimator $(\widehat{b}_{1N},\widehat{b}_{2N},\widehat{\sigma}_{N})$ is defined as the maximizer of the function
\begin{equation*}
S_N(b_1, b_2, \sigma) = \frac{1}{N} \sum_{i=1}^N \log f(X_{i2}\vert X_{i1}, T_{i1}, T_{i2}) \ind\{X_{i1} \leq \delta_0\},
\end{equation*}
where the maximum is taken over all truncated functions of the form~\eqref{eq:ExpansionTruncated}. The corresponding population version is given by
\begin{equation*}
    S(b_1, b_2, \sigma) = \expv[S_N(b_1, b_2, \sigma)],
\end{equation*}
where we note that, hereinafter, all expectations and probabilities refer to the true model induced by $(b_{10}, b_{20}, \sigma_0)$. To prove that the maximizer of $S_N$ converges to the (truncated) maximizer of~$S$ in Lemma~\ref{lem:Conv}, we use Theorem~3.4.1 of \cite{van2023weak} which is stated below. The theorem essentially requires the population function $S$ to decrease quadratically away from its maximizer (see \eqref{eq:Bias}) and the random fluctuations of the empirical function $S_N$ around $S$ must remain small by comparison (see \eqref{eq:Variance}). For consistency with our framework, we slightly adapted the theorem.

\begin{theorem}[\cite{van2023weak}]\label{thm:VdV}
For each $N$, let $S_N$ and $S$ denote stochastic processes indexed by a set $\Theta_N$ and $\theta_{0N} \in \Theta_N$. Let~$\theta \mapsto d(\theta, \theta_{0N})$ be an arbitrary map from $\Theta_N$ to $[0, \infty)$, and $\underline{r}_N \ge 0$. Suppose that, for every $N$ and $r > \underline{r}_N$,
\begin{equation}\label{eq:Bias}
\inf_{\theta \in \Theta_N : r/2 < d(\theta, \theta_{0N}) \leq r} S(\theta_{0N}) -  S(\theta)  \geq r^2,
\end{equation}
\begin{equation}\label{eq:Variance}
\expv \left[\sup_{\theta \in \Theta_N: d(\theta, \theta_{0N}) \le r} \sqrt{N} \abs*{ (S_N - S)(\theta) - (S_N - S)(\theta_{0N}) } \right]\lesssim \phi_N(r),
\end{equation}
for increasing functions $\phi_N\colon [\underline{r}_N, \infty) \to \mathbb{R}$ such that $r \mapsto \phi_N(r)/r^\alpha$ is decreasing for some $\alpha < 2$. Let $r_N$ satisfy
\begin{align*}
    \phi_N(r_N) \le \sqrt{N}r_N^2, && r_N \ge \underline{r}_N.
\end{align*}
\noindent If the sequence $\widehat{\theta}_N$ maximizes $S_N(\theta)$ over $\Theta_N$, then $d(\widehat{\theta}_N, \theta_{0N}) = O_p(r_N)$.
\end{theorem}
We consider the distance
\begin{equation}\label{eq:distance}
d((b_{1}, b_{2}, \sigma), (\underline{b}_{1}, \underline{b}_{2}, \underline{\sigma})) = \sqrt{\normw{b_1 - \underline{b}_1}^2 +\normww{b_2 - \underline{b}_2}^2 + \normwww{\sigma^2 - \underline{\sigma}^2}^2}.
\end{equation}

\begin{lemma}\label{lem:Conv}
It holds that
\begin{equation*}
\normw{\widehat{b}_{1N} - b_{10N}}^2 +\normww{\widehat{b}_{2N} - b_{20N}}^2 + \normwww{\widehat{\sigma}^2_N-\sigma_{0N}^2}^2 = O_p(\log(M) MN^{-1}).
\end{equation*}
\end{lemma}

\begin{proof}[Proof of Lemma~\ref{lem:Conv}]
We use Theorem~\ref{thm:VdV}. To this end, identify $\theta$ with $(b_1, b_2, \sigma)$ and  set $\theta_{0N} = (b_{10N}, b_{20N}, \sigma_{0N})$. The set $\Theta_N$ refers to the set of functions in \eqref{eq:ExpansionTruncated} constrained by Assumption~\ref{ass:beta}. Consider the distance $d$ in \eqref{eq:distance}. Define $\underline{r}_N = 2c_0\,d(\theta_{0N}, \theta_0)$ for the choice of $c_0 > 0$ in Lemma~\ref{lem:Bias}. The lemma then implies that the condition~\eqref{eq:Bias} is satisfied. Moreover, Lemma~\ref{lem:Variance} shows that $\phi_N(r_N) \lesssim \sqrt{N} r_N^2$ if $\sqrt{M\log(M)}\,r_N \leq \sqrt{N} r_N^2$. The convergence rate is thus determined by
$$ \sqrt{M\log(M)}\,r_N \leq \sqrt{N}r_N^2, \qquad r_N \geq 2c_0\,d(\theta_{0N},\theta_0).$$
Using $\normw{\cdot} \leq \normwww{\cdot} \leq \norm{\cdot}_{L^2}$ by Lemma~\ref{lem:UpperBoundNorm}, Assumption~\ref{ass:beta}, and $M \sim N^{1/2\beta}$,
\begin{equation*}
d^2(\theta_{0N}, \theta_0) \lesssim \sum_{k=M+1}^\infty (b_{1,0k}^2 + b_{2,0k}^2+ \sigma_{0k}^2) \lesssim \sum_{k=M+1}^\infty k^{-2\beta} \lesssim M^{1-2\beta} = MN^{-1}.
\end{equation*}
Combining the above, we obtain the rate $d^2(\widehat{\theta}_N, \theta_{0N}) = O_p(\log(M) MN^{-1})$.
\end{proof}

\subsection{Bias term}

Identify $\theta$ with $(b_1, b_2, \sigma)$ and observe that the difference
\begin{equation}\label{eq:KLdiv}
    S(\theta_0) - S(\theta) = \expv\left[\log \frac{f_0(X_2\vert X_1, T_1, T_2)}{f(X_2\vert X_1, T_1, T_2)} \,\ind\{X_1 \leq \delta_0\}\right],
\end{equation}
relates to a truncated version of the Kullback-Leibler divergence. Our first goal is to bound $S(\theta_{0N})-S(\theta)$.

\begin{lemma}[Bias]\label{lem:Bias}
Consider the distance $d$ in \eqref{eq:distance}. There exists some constant $c_0 > 0$ such that
\begin{equation*}
S(\theta_{0N})-S(\theta) \gtrsim d^2(\theta, \theta_{0N}),
\end{equation*}
provided $d(\theta, \theta_{0N}) \geq c_0\, d(\theta_0, \theta_{0N})$.
\end{lemma}

\begin{proof}
The inequalities in Lemma~\ref{lem:KLlower} yield
\begin{equation*}
    S(\theta_{0N}) - S(\theta) = S(\theta_0) - S(\theta) + S(\theta_{0N}) - S(\theta_0) \geq c_2 d^2(\theta, \theta_0) - c_1d^2(\theta_{0N}, \theta_0).
\end{equation*}
The reverse triangle inequality implies
\begin{equation*}
    d(\theta,\theta_0) \geq d(\theta, \theta_{0N}) - d(\theta_0, \theta_{0N})\geq (1-c_0^{-1}) d(\theta, \theta_{0N}),
\end{equation*}
provided $d(\theta, \theta_{0N}) \geq c_0\, d(\theta_0, \theta_{0N})$ for some $c_0 > 1$. Combining the above, we obtain
\begin{equation*}
    S(\theta_{0N}) - S(\theta) \geq (c_2  (1-c_0^{-1})^2 - c_1 c_0^{-2}) d^2(\theta, \theta_{0N}).
\end{equation*}
Choosing $c_0$ large enough such that $c_2  (1-c_0^{-1})^2 > c_1 c_0^{-2}$ then finishes the proof.
\end{proof}

\begin{lemma}\label{lem:KLlower}
Consider the distance $d$ in \eqref{eq:distance}. There exist constants $c_1, c_2 > 0$ such that
\begin{equation*}
    c_1 d^2(\theta_0, \theta) \geq S(\theta_0) - S(\theta) \geq c_2 d^2(\theta_0, \theta).
\end{equation*}
\end{lemma}

\begin{proof}

Inserting Gaussian densities into \eqref{eq:KLdiv}, one observes that
{\small
\begin{align*}
    &S(\theta_0) - S(\theta) \\
    &= 
    \frac{1}{2}
    \expv \left[\left(-\log \frac{v_0(T_1, T_2)}{v(T_1, T_2)} + \frac{v_0(T_1, T_2)}{v(T_1, T_2)} - 1+\frac{(m_0(X_1, T_1, T_2) - m(X_1, T_1, T_2))^2}{v(T_1, T_2)}\right)\ind\{X_1 \leq \delta_0\} \right].
\end{align*}}

\noindent
Note that $C_1^{-1} (T_2 - T_1) \leq v(T_1, T_2) \leq C_1(T_2-T_1)$ for some constant $C_1 >0$. The first inequality $d^2(\theta_0, \theta) \gtrsim S(\theta_0) - S(\theta)$ then follows from Lemma~\ref{lem:MeanDiffUpper} and Lemma~\ref{lem:VarDiffUpper} (note that $-\log(x) +x -1 \geq 0$ for $x > 0$). For the second one, we combine Lemma~\ref{lem:MeanDiff} and Lemma~\ref{lem:VarDiff} with the  existence of a constant $C_2 > 0$ such that $\prob(X_1 \leq \delta_0\vert T_1)\geq C_2$. This yields
\begin{align*}
     S(\theta_0) - S(\theta) &\gtrsim (C_1^{-1}c_3 +C_2c_4(1-\varepsilon^{-1})) \,\normw{b_{10}-b_1}^2 + C_1^{-1}c_3 \, \normww{b_{20}-b_2}^2\\
     & \qquad \qquad  + C_2c_5 (1-\varepsilon) \,\normwww{\sigma_0^2 - \sigma^2}^2,
\end{align*}
for arbitrary $\varepsilon \in (0,1)$. Choosing $\varepsilon$ such that $C_1^{-1}c_3 + C_2c_4(1-\varepsilon^{-1})>0$ finishes the proof of the lemma.
\end{proof}

\begin{remark}
The lemma above establishes tight bounds on a truncated version of the Kullback--Leibler divergence. It further implies that two randomly located censored measurements per independent replicate suffice to identify the model parameters.
\end{remark}

\subsection{Variance}

\begin{lemma}[Variance]\label{lem:Variance} Identify $\theta$ with $(b_1, b_2, \sigma)$. It holds that
\begin{align*}
\expv \left[ \sup_{\substack{\theta \in \Theta_N, \\ d(\theta,\theta_{0N}) \leq r}} \sqrt{N} \abs{(S_N-S)(\theta)-(S_N-S)(\theta_{0N})}\right] \lesssim  \phi_N(r),
\end{align*}
where $\phi_N(r) = \sqrt{M\log(M)} \,r\left(1 + \frac{\sqrt{M\log(M)}}{r \sqrt{N}}\right)$.
\end{lemma}

\begin{proof}
We need to derive a bound for the empirical process indexed by the class of functions
\begin{equation*}
\mathcal{M}_{N, r}^\kappa = \left\lbrace \kappa (\ell_\theta - \ell_{\theta_{0N}})\colon \theta \in \Theta_N, d(\theta,\theta_{0N}) \leq r\right\rbrace,
\end{equation*}
where $\ell_{\theta}(X_1, X_2, T_1,T_2) = \log f_{\theta}(X_2\vert X_1, T_1, T_2) \,\ind\{X_1 \leq \delta_0\}$. Lemma~\ref{lem:FG} yields
\begin{equation}\label{eq:Lipschitz}
\abs*{\ell_\theta(X_1, X_2, T_1,T_2)-\ell_{\underline{\theta}}(X_1, X_2, T_1,T_2)} \leq F(X_1, X_2, T_1, T_2) \, G(\theta, \underline{\theta};T_1, T_2),
\end{equation}
for given functions $F$ and $G$. Let $C_1>\max\{G(\theta, \underline{\theta};T_1, T_2), 1\}$ and choose $\kappa > 0$ to be small enough such that
\begin{equation*}
\expv\left[\exp\left(\kappa C_1 F(X_1, X_2, T_1, T_2)\right) \,\big\vert \, T_1, T_2 \right] < C_2,
\end{equation*}
almost surely for some constant $C_2>0$. In the following, we consider the \textit{Bernstein distance} \cite[Section 2.14.2]{van2023weak} which is given by 
\begin{align*}
\norm{\ell}_{\text{B}} = \left(2 \, \expv\left[e^{\abs{\ell(X_1, X_2, T_1, T_2)}}-1-\abs{\ell(X_1, X_2, T_1, T_2)}\right]\right)^{1/2},
\end{align*}
 Using Lemma~\ref{lem:FG} and Lemma~\ref{lem:G}, we obtain
 \begin{equation}\label{eq:EstBernstein}
\begin{aligned}
\norm{\kappa (\ell_\theta - \ell_{\underline{\theta}})}^2_{\text{B}} &\lesssim
\expv\left[\sum_{k=2}^\infty \frac{\kappa^k F(X_1, X_2, T_1,T_2)^kG(\theta, \underline{\theta};T_1, T_2)^k}{k!}\right] \\
&\leq \expv\left[G(\theta, \underline{\theta};T_1, T_2)^2 \,\expv\left\lbrace \sum_{k=2}^\infty \frac{\kappa^kC_1^kF(X_1, X_2, T_1,T_2)^k}{k!}\, \Big\vert \, T_1, T_2\right\rbrace \right] \\
&\leq \expv\left[G(\theta, \underline{\theta};T_1, T_2)^2 \,\expv\left\lbrace \exp\left(\kappa C_1 F(X_1, X_2, T_1, T_2)\right) \,\Big\vert \, T_1, T_2 \right\rbrace \right] \\
&\leq C_2 \expv\left[G(\theta, \underline{\theta};T_1, T_2)^2 \right] \lesssim \,d^2(\theta,\underline{\theta}).
\end{aligned}
\end{equation}
This shows that the Bernstein norm is bounded by a multiple of the distance $d$. Theorem~2.4.18' of \cite{van2023weak} then yields
{\small\begin{equation*}
\expv \left[ \sup_{\substack{\theta \in \Theta_N, \\ d(\theta,\theta_{0N}) \leq r}} \sqrt{N} \abs{(S_N-S)(\theta)-(S_N-S)(\theta_{0N})}\right]
\lesssim \tilde{J}_{[\,]}(r, \mathcal{M}_{N,r}^\kappa, \norm{\cdot}_{\text{B}}) \left(1+\frac{\tilde{J}_{[\,]}(r, \mathcal{M}_{N,r}^\kappa, \norm{\cdot}_{\text{B}})}{r^2\sqrt{N}}\right),
\end{equation*}}
where
\begin{equation*}
\tilde{J}_{[\,]}(r,  \mathcal{F}, \norm{\cdot}) = \int_0^r \sqrt{1 + \log N_{[\,]}(\varepsilon,\mathcal{F}, \norm{\cdot})} \, \diff \varepsilon,
\end{equation*}
denotes the \textit{bracketing integral} of a class $\mathcal{F}$ with respect to a norm $\norm{\cdot}$ \citep[p.~132, p.~338]{van2023weak}. We need to bound the \textit{bracketing number} $N_{[\,]}(\varepsilon , \mathcal{M}_{N,r}^\kappa, \norm{\cdot}_{\text{B}})$ of the class~$\mathcal{M}_{N,r}^\kappa$ relative to the Bernstein norm. Consider brackets $[b_{1L}, b_{1U}]$, $[b_{2L}, b_{2U}]$, and $[\sigma_{L}, \sigma_{U}]$ of $\norm{\cdot}_\infty$-size $\varepsilon/3$ which implies
$\norm{b_{1U} - b_{1L}}_\infty + \norm{b_{2U}-b_{2L}}_\infty + \norm{\sigma_U - \sigma_L}_\infty < \varepsilon$. Using~\eqref{eq:Lipschitz} and Lemma~\ref{lem:G}, we construct a corresponding bracket for $\kappa\ell_\theta$ by
\begin{equation*}
    \kappa\ell_{\theta_L}(X_1, X_2, T_1, T_2) \pm \, \varepsilon\kappa F(X_1, X_2, T_1, T_2),
\end{equation*}
where $\theta_L = (b_{1L}, b_{2L}, \sigma_L)$. This bracket covers a function $\kappa \ell_\theta$ if
\begin{equation*}
\theta \in [b_{1L}, b_{1U}] \times [b_{2L}, b_{2U}] \times 
[\sigma_{L}, \sigma_{U}].
\end{equation*}
Using similar arguments as in~\eqref{eq:EstBernstein}, one can show that the $\norm{\cdot}_{\text{B}}$-size of the brackets is bounded by $C_3\varepsilon$ for some constant $C_3 > 0$\footnote{Note that we may safely assume that $\varepsilon$ is bounded by a constant, since for large $\varepsilon$, a single bracket suffices to cover the parameter space, yielding a negligible contribution of $O(r)$ to the bracketing integral.}. Therefore, an upper bound for $N_{[\,]}(C_3 \varepsilon, \mathcal{M}_{N,r}^\kappa, \norm{\cdot}_B)$ is given by
\begin{equation*}
        N_{[\,]}(\varepsilon, B_{N,r}^{b_1}, \norm{\cdot}_\infty) \times N_{[\,]}(\varepsilon, B_{N,r}^{b_2}, \norm{\cdot}_\infty) \times N_{[\,]}(\varepsilon, B_{N,r}^{\sigma}, \norm{\cdot}_\infty),
\end{equation*}
where $B_{N,r}^{\sigma} = \{\sigma = \exp\big(\sum_{k=1}^M \sigma_k \varphi_k\big)\colon \normwww{\sigma^2-\sigma_{0N}^2} \leq r\}$ and $B^{b_j}_{N,r}$ are defined correspondingly. Furthermore, $N_{[\,]}(\varepsilon, B_{N,r}^{\sigma}, \norm{\cdot}_\infty)$ denotes the $\varepsilon$-bracketing number of the ball $B_{N,r}^{\sigma}$ with respect to the supremum norm~$\norm{\cdot}_\infty$. To obtain the required bounds, we proceed via a volumetric argument. Using Lemma~\ref{lem:LowerBoundNorm} and Lemma~\ref{lem:LowerBoundNormII} alongside a Lipschitz condition, we bound the bracketing numbers by the covering numbers of Euclidean balls. Applying the volumetric bound from Corollary~4.2.11 in \cite{vershynin2018high} then yields
\begin{align*}
    \log N_{[\,]}(\varepsilon , \mathcal{M}_{N,r}^\kappa, \norm{\cdot}_{\text{B}}) \leq C_5 M \log\left(1+\frac{C_4r M^2}{\varepsilon}\right), && \varepsilon > 0,
\end{align*}
for some constants $C_4,C_5 > 0$. 
This finally implies
\begin{equation*}
    \tilde{J}_{[\,]}(r, \mathcal{M}_{N,r}^\kappa, \norm{\cdot}_\textup{B}) \leq \int_0^r \sqrt{1 + C_5 M \log\left(1+\frac{C_4r M^2}{\varepsilon}\right)} \, \diff \varepsilon \lesssim  \sqrt{M\log(M)} \,r ,
\end{equation*}
which concludes the proof of the lemma.
\end{proof}

\subsection{Auxiliary results}

\begin{lemma}\label{lem:MeanDiff}
    There exists some $c_3 > 0$ such that
    \begin{equation*}
    \expv \left[\frac{(m_0(X_1, T_1, T_2) - m(X_1, T_1, T_2))^2}{T_2-T_1} \ind\{X_1 \leq \delta_0\}\right] \geq c_3 \left(\normw{b_{10}-b_1}^2  + \normww{b_{20}-b_2}^2\right).
    \end{equation*}
\end{lemma}

\begin{proof}
The numerator is given by
\begin{equation*}
    \begin{aligned}
    &m_0(X_1, T_1, T_2) - m(X_1, T_1, T_2) \\
    &\qquad = X_1 (\Phi_0(T_1, T_2) - \Phi(T_1, T_2)) +\int_{T_1}^{T_2} \{\Phi_0(u, T_2) b_{20}(u) - \Phi(u, T_2) b_2(u) \} \,\diff u.
\end{aligned}
\end{equation*}
Centering the above expression using $\mu_{\delta_0}(T_1) = \expv[X_1\vert X_1 \leq {\delta_0}, T_1]$, we obtain
\begin{align*}
    &m_0(X_1, T_1, T_2) - m(X_1, T_1, T_2) = (X_1-\mu_{\delta_0}(T_1)) (\Phi_0(T_1, T_2) - \Phi(T_1, T_2)) \\
    &\qquad + \mu_{\delta_0}(T_1) (\Phi_0(T_1, T_2) - \Phi(T_1, T_2)) +\int_{T_1}^{T_2} \{\Phi_0(u, T_2) b_{20}(u) - \Phi(u, T_2) b_2(u) \} \,\diff u.
\end{align*}
Squaring the above and taking the conditional expectation with respect to~$(X_1 \leq {\delta_0}, T_1, T_2)$, the cross-term involving $(X_1-\mu_{\delta_0}(T_1))$ vanishes. Since $\var(X_1\vert X_1 \leq \delta_0, T_1)>C_1$ for some constant $C_1 > 0$, we obtain
\begin{align*}
    &\expv \left[(m_0(X_1, T_1, T_2) - m(X_1, T_1, T_2))^2 \, \big\vert X_1 \leq \delta_0, T_1, T_2 \right] \geq C_1  (\Phi_0(T_1, T_2) - \Phi(T_1, T_2))^2 \\
    &\quad + \left( \mu_{\delta_0}(T_1) (\Phi_0(T_1, T_2) - \Phi(T_1, T_2)) +\int_{T_1}^{T_2} \{\Phi_0(u, T_2) b_{20}(u) - \Phi(u, T_2) b_2(u) \} \,\diff u \right)^2.
\end{align*}
As a next step, we apply the inequality 
\begin{align}\label{iq:PeterPaul}
(x + y)^2 \geq (1-\varepsilon^{-1}) x^2 + (1-\varepsilon)y^2, && \varepsilon \in (0,1),
\end{align}
to obtain a lower bound for the second summand. Since there exist constants $C_2, C_3 >0$ such that $\mu_{\delta_0}(T_1)^2 \leq C_2 $ and $\prob(X_1 \leq \delta_0 \vert T_1) \geq C_3$, we obtain
\begin{equation*}
    \expv \left[\frac{(m_0(X_1, T_1, T_2) - m(X_1, T_1, T_2))^2}{T_2-T_1}\,\ind\{X_1 \leq \delta_0\}\right]  \geq C_3 \big(C_1 + C_2(1-\varepsilon_1^{-1})\big)\, R_1 + C_3(1-\varepsilon_1)\, R_{23},
\end{equation*}
for the terms
\begin{align*}
    R_1 &= \expv\left[\frac{1}{T_2 - T_1}(\Phi_0(T_1, T_2) - \Phi(T_1, T_2))^2\right], \\
    R_{23} &= \expv \left[\frac{1}{T_2 - T_1} \left(\int_{T_1}^{T_2} \{\Phi_0(u, T_2) b_{20}(u) - \Phi(u, T_2) b_2(u) \} \,\diff u\right)^2\right],
\end{align*}
which are treated in Lemma~\ref{lem:Rs}. An appropriate choice of $\varepsilon_1\in (0,1)$ close to 1 yields a constant $C_4>0$ such that $C_3 \big(C_1 + C_2(1-\varepsilon_1^{-1})\big) \geq C_4$ and $C_3(1-\varepsilon_1) \geq C_4$. Thus,
\begin{equation}\label{eq:BiasMean}
    \expv \left[\frac{(m_0(X_1, T_1, T_2) - m(X_1, T_1, T_2))^2}{T_2-T_1}\,\ind\{X_1 \leq \delta_0\}\right]  \geq C_4 (R_1 + R_{23}).
\end{equation}
Lemma~\ref{lem:Rs} yields a constant $C_5>0$ such that $R_1 \geq C_5 \normw{b_{10}-b_1}^2$. Concerning $R_{23}$, we apply inequality \eqref{iq:PeterPaul} for some appropriate choice of $\varepsilon_2 \in (0,1)$ to obtain
\begin{equation}\label{eq:BiasMeanII}
    R_{23} \geq  (1-\varepsilon_2^{-1})\,R_2 + (1-\varepsilon_2) \,R_3,
\end{equation}
where
\begin{align*}
    R_2 &= \expv\left[\frac{1}{T_2 - T_1}\left(\int_{T_1}^{T_2} \{\Phi_0(u, T_2) - \Phi(u,T_2)\}b_{2}(u)\,\diff u\right)^2\right], \\
    R_3 &= \expv\left[\frac{1}{T_2 - T_1}\left(\int_{T_1}^{T_2} \Phi_0(u, T_2) \{b_{20}(u)- b_{2}(u)\}\,\diff u\right)^2\right].
\end{align*}
Lemma~\ref{lem:Rs} yields constants $C_6, C_7 > 0$ such that $R_2 \leq C_6 \normw{b_{10}-b_1}^2$ and $R_3 \geq C_7\normww{b_{20}-b_2}^2$. Substituting the bounds for $R_1$, $R_2$, and $R_3$ back into \eqref{eq:BiasMean} and \eqref{eq:BiasMeanII}, we get
\begin{align*}
    &\expv \left[\frac{(m_0(X_1, T_1, T_2) - m(X_1, T_1, T_2))^2}{T_2-T_1}\ind\{X_1 \leq \delta_0\}\right]\\
    &\qquad \qquad \geq C_4 \left( C_5 + C_6(1-\varepsilon_2^{-1}) \right)\normw{b_{10}-b_1}^2 + C_4 C_7(1-\varepsilon_2) \normww{b_{20}-b_2}^2.
\end{align*}
Choosing $\varepsilon_2 \in (0,1)$ such that $C_4(C_5+C_6(1-\varepsilon_2^{-1})) > c_3$ and $C_4 C_7 (1-\varepsilon_2) > c_3$ finishes the proof.
\end{proof}

\begin{lemma}\label{lem:MeanDiffUpper}
    It holds that
    \begin{equation*}
    \expv \left[\frac{(m_0(X_1, T_1, T_2) - m(X_1, T_1, T_2))^2}{T_2-T_1}\right] \lesssim \normw{b_{10}-b_1}^2  + \normww{b_{20}-b_2}^2.
    \end{equation*}
\end{lemma}

\begin{proof}
    The proof is similar to the one of Lemma~\ref{lem:MeanDiff}.
\end{proof}

\begin{lemma}\label{lem:VarDiff}
    There exist $c_4, c_5 > 0$ such that
    \begin{equation*}
    \expv \left[-\log \frac{v_0(T_1, T_2)}{v(T_1, T_2)} + \frac{v_0(T_1, T_2)}{v(T_1, T_2)} - 1\right] \geq c_4 (1-\varepsilon^{-1}) \normw{b_{10}-b_1}^2 +c_5 (1-\varepsilon) \normwww{\sigma_0^2-\sigma^2}^2,
    \end{equation*}
    for any $\varepsilon \in (0,1)$.
\end{lemma}

\begin{proof}
For any $C_1>0$, there exists a constant $C_2$ such that $-\log(x) + x - 1 \geq C_2(x-1)^2$ for every $\abs{x} < C_1$. Since $\abs{v_0(T_1, T_2)/v(T_1, T_2)}< C_1$ by our assumptions, we obtain
\begin{equation*}
     \expv \left[-\log \frac{v_0(T_1, T_2)}{v(T_1, T_2)} + \frac{v_0(T_1, T_2)}{v(T_1, T_2)} - 1\right] \geq C_3  \, \expv \left[\frac{(v_0(T_1, T_2)-v(T_1, T_2))^2}{(T_2-T_1)^2}\right].
\end{equation*}
Rewriting the numerator, we obtain
\begin{align*}
    &v_0(T_1, T_2) - v(T_1, T_2) \\
    &\qquad = \int_{T_1}^{T_2} \{\Phi_0^2(u,T_2)-\Phi^2(u,T_2)\}\sigma^2(u)\, \diff u + \int_{T_1}^{T_2} \Phi_0^2(u,T_2)\{\sigma_0^2(u)-\sigma^2(u)\}\, \diff u.
\end{align*}
For $\varepsilon \in (0,1)$, the inequality in~\eqref{iq:PeterPaul} yields
\begin{equation*}
    \expv \left[\frac{(v_0(T_1, T_2)-v(T_1, T_2))^2}{(T_2-T_1)^2}\right] \geq  (1-\varepsilon^{-1}) \,R_4 + (1-\varepsilon)\,R_5,
\end{equation*}
where
\begin{align*}
    R_4 &= \expv \left[\frac{1}{(T_2-T_1)^2}\left(\int_{T_1}^{T_2} \{\Phi_0^2(u,T_2)-\Phi^2(u,T_2)\}\sigma^2(u)\, \diff u\right)^2\right], \\
   R_5 &= \expv \left[\frac{1}{(T_2-T_1)^2}\left(\int_{T_1}^{T_2} \Phi_0^2(u,T_2)\{\sigma_0^2(u)-\sigma^2(u)\}\, \diff u\right)^2\right].
\end{align*}
Invoking Lemma~\ref{lem:Rs}, there exist constants $C_4, C_5 > 0$ such that $R_4 \leq C_4 \, \normw{b_{10} - b_1}^2$
as well as $R_5 \geq C_5 \, \normwww{\sigma_0^2 - \sigma^2}^2$. Combining the above,
\begin{align*}
    &\expv \left[-\log \frac{v_0(T_1, T_2)}{v(T_1, T_2)} + \frac{v_0(T_1, T_2)}{v(T_1, T_2)} - 1\right] \\
    & \qquad \qquad \geq  C_3 C_4 (1-\varepsilon^{-1}) \normw{b_{10}-b_1}^2 +C_3 C_5 (1-\varepsilon) \normwww{\sigma_0^2-\sigma^2}^2,
\end{align*}
which finishes the proof of the lemma.
\end{proof}
\begin{lemma}\label{lem:VarDiffUpper}
    It holds that
    \begin{equation*}
    \expv \left[-\log \frac{v_0(T_1, T_2)}{v(T_1, T_2)} + \frac{v_0(T_1, T_2)}{v(T_1, T_2)} - 1\right] \lesssim \normw{b_{10}-b_1}^2 + \normwww{\sigma_0^2-\sigma^2}^2.
    \end{equation*}
\end{lemma}
\begin{proof}
    The proof is similar to the one of Lemma~\ref{lem:VarDiff}.
\end{proof}

\begin{lemma}\label{lem:Rs}
It holds that
\begin{align*}
    R_1 &= \expv\left[\frac{1}{T_2 - T_1}(\Phi(T_1, T_2) - \underline{\Phi}(T_1, T_2))^2\right] &&\simeq \normw{b_{1}-\underline{b}_1}^2,\\
    R_2 &= \expv\left[\frac{1}{T_2 - T_1}\left(\int_{T_1}^{T_2} \{\Phi(u, T_2) - \underline{\Phi}(u,T_2)\}b_{2}(u)\,\diff u\right)^2\right] &&\lesssim \normw{b_{1}-\underline{b}_1}^2,\\
    R_3 &= \expv\left[\frac{1}{T_2 - T_1}\left(\int_{T_1}^{T_2} \Phi(u, T_2) \{b_{2}(u)- \underline{b}_{2}(u)\}\,\diff u\right)^2\right] &&\simeq \normww{b_{2}-\underline{b}_2}^2,\\
    R_4 &= \expv \left[\frac{1}{(T_2-T_1)^2}\left(\int_{T_1}^{T_2} \{\Phi^2(u,T_2)-\underline{\Phi}^2(u,T_2)\}\sigma^2(u)\, \diff u\right)^2\right] &&\lesssim \normw{b_{1}-\underline{b}_1}^2, \\
    R_5 &= \expv \left[\frac{1}{(T_2-T_1)^2}\left(\int_{T_1}^{T_2} \Phi^2(u,T_2)\{\sigma^2(u)-\underline{\sigma}^2(u)\}\, \diff u\right)^2\right] &&\simeq \normwww{\sigma^2 - \underline{\sigma}^2}^2.
\end{align*}
\end{lemma}

\begin{proof}
    Concerning $R_1$, an application of the mean value theorem yields a constant $C_1 > 0$ with
\begin{align*}
    \expv\left[\frac{1}{T_2 - T_1}\left(\Phi(T_1, T_2) - \underline{\Phi}(T_1, T_2)\right)^2\right] &\geq C_1\, \expv\left[\frac{1}{T_2 - T_1}\left(\int_{T_1}^{T_2} b_{1}(u)-\underline{b}_1(u)\, \diff u\right)^2\right] \\
    &= 2\, C_1 \, \normw{b_{1}-\underline{b}_1}^2.
\end{align*}
Analogously, it follows that $R_1 \lesssim \normw{b_{1}-\underline{b}_1}^2$ and thus $R_1 \simeq \normw{b_{1}-\underline{b}_1}^2$.

Concerning $R_2$, applying the Cauchy-Schwarz inequality, the boundedness of $b_2$, the mean value theorem, and swapping the order of integration yields
\begin{align*}
   R_2 & \leq C_2 \,\expv\left[\frac{T_2 - T_1}{T_2 - T_1}\int_{T_1}^{T_2} \{\Phi(u, T_2) - \underline{\Phi}(u,T_2)\}^2\,\diff u\right]\\
    &\leq C_3 \,\expv\left[\int_{T_1}^{T_2} \left(\int_u^{T_2} b_{1}(v) - \underline{b}_1(v) \,\diff v\right)^2\,\diff u\right] \\
    &= 2 \,C_3 \,\int_0^1 \int_0^{t_2} \int_{t_1}^{t_2} \left(\int_u^{t_2} b_{1}(v)-\underline{b}_1(v)\, \diff v\right)^2 \, \diff u \,\diff t_1 \, \diff t_2 \\
    &= 2 \,C_3 \,\int_0^1 \int_0^{t_2} \left(\int_0^u  \diff t_1\right) \left(\int_u^{t_2} b_{1}(v)-\underline{b}_1(v)\, \diff v\right)^2  \,\diff u \, \diff t_2 \\
    &\leq 2 \,C_3 \,\int_0^1 \int_0^{t_2}  \frac{u}{t_2-u} \left(\int_u^{t_2} b_{1}(v)-\underline{b}_1(v)\, \diff v\right)^2  \,\diff u \, \diff t_2 \\
    &\lesssim 2 C_3 \,\normw{b_{1}-\underline{b}_1}^2.
\end{align*}

Concerning $R_3$, we first use integration by parts, 
\begin{align*}
&\int_{T_1}^{T_2} \Phi(u, T_2)^{-1}\Phi(u, T_2) \{b_{2}(u)- \underline{b}_{2}(u)\}\,\diff u  = \int_{T_1}^{T_2} \Phi(u,T_2) \{b_{2}(u) - \underline{b}_2(u)\}\, \diff u \\
&\qquad \qquad - \int_{T_1}^{T_2} b_{1}(u) \Phi(u,T_2)^{-1} \int_{T_1}^u \Phi(v,T_2) \{b_{2}(v) - \underline{b}_2(v)\} \, \diff v \, \diff u.
\end{align*}
Squaring both sides and using Cauchy-Schwarz as well as $\Phi(v, T_2)=\Phi(v,u)\Phi(u, T_2)$,
\begin{equation}\label{eq:IBP}
\begin{aligned}
   &\left(\int_{T_1}^{T_2} b_{2}(u)- \underline{b}_{2}(u)\,\diff u \right)^2 \leq C_4 \left(\int_{T_1}^{T_2} \Phi(u,T_2) \{b_{2}(u) - \underline{b}_2(u)\}\, \diff u\right)^2 \\
   &\qquad \qquad + C_4 (T_2 - T_1) \int_{T_1}^{T_2}\left(\int_{T_1}^u \Phi(v,u) \{b_{2}(v) - \underline{b}_2(v)\} \, \diff v  \right)^2\, \diff u.
   \end{aligned}
\end{equation}
Then,
\begin{align*}
    &\expv\left[\int_{T_1}^{T_2}\left(\int_{T_1}^u \Phi(v,u) \{b_{2}(v) - \underline{b}_2(v)\} \, \diff v  \right)^2\, \diff u\right] \\
    &\qquad \qquad  = 2 \int_0^1 \int_{t_1}^1 \int_{t_1}^{t_2} \left(\int_{t_1}^u \Phi(v,u)\{b_{2}(v) - \underline{b}_2(v) \}\, \diff v \right)^2 \diff u \, \diff {t_2} \, \diff {t_1} \\
    &\qquad \qquad \leq 2 \int_0^1 \int_{t_1}^1 \left(\int_{t_1}^u \Phi(v,u)\{b_{2}(v) - \underline{b}_2(v)\} \, \diff v \right)^2 \diff u\, \diff {t_1}\\
    &\qquad \qquad \leq 2 \int_0^1 \int_{t_1}^1 \frac{1}{t_2-t_1} \left(\int_{t_1}^{t_2} \Phi(v,t_2)\{b_{2}(v) - \underline{b}_2(v)\} \, \diff v \right)^2 \diff t_2\, \diff {t_1}\\
    &\qquad \qquad = C_5 \,\expv\left[\frac{1}{T_2 - T_1}\left(\int_{T_1}^{T_2} \Phi(u, T_2) \{b_{2}(u)- \underline{b}_{2}(u)\}\,\diff u\right)^2\right].
\end{align*}
Combining this bound with the expectation of \eqref{eq:IBP} guarantees that 
\begin{equation*}
    R_3 \geq C_6 \,\expv\left[\frac{1}{T_2 - T_1}\left(\int_{T_1}^{T_2} b_{2}(u)- \underline{b}_{2}(u)\,\diff u \right)^2\right] \geq C_6 \,\normww{b_{2}-\underline{b}_2}^2,
\end{equation*}
which reads as $R_3 \gtrsim \normww{b_{2}-\underline{b}_2}^2$. Using similar arguments, one can show that $R_3 \lesssim \normww{b_{2}-\underline{b}_2}^2$.

Concerning $R_4$ and $R_5$, the proofs are similar to the ones of $R_2$ and $R_3$.
\end{proof}

\begin{lemma}\label{lem:FG}
Let $\ell_\theta(X_1, X_2, T_1, T_2) = \log f_\theta(X_2\vert X_1, T_1, T_2) \ind\{X_1 \leq \delta_0\}$ for the generic parameter $\theta = (b_1, b_2, \sigma)$. 
Then, there exists a function $F$ such that
\begin{align*}
    \abs{\ell_\theta(X_1, X_2, T_1, T_2)-\ell_{\underline{\theta}}(X_1, X_2, T_1, T_2)} \lesssim F(X_1, X_2, T_1, T_2)\, G(\theta, \underline{\theta};T_1, T_2),&& \forall \,\theta, \underline{\theta},
\end{align*}
where $G$ is given by
\begin{align*}
    & G(\theta, \underline{\theta};T_1, T_2) \\
    &\qquad = \frac{\abs{\Phi(T_1, T_2) - \underline{\Phi}(T_1, T_2)}}{\sqrt{T_2 - T_1}} + \frac{\abs{m(0,T_1, T_2) - \underline{m}(0, T_1, T_2)}}{\sqrt{T_2 - T_1}} + \frac{\abs{v(T_1, T_2) - \underline{v}(T_1, T_2)}}{T_2 - T_1}.
\end{align*}
Furthermore, for any $C_1>0$, there exists some $\kappa > 0$ sufficiently small and $C_2 > 0$ sufficiently large, such that $F$ satisfies
\begin{equation*}
    \expv\left[\exp\{\kappa C_1F(X_1, X_2, T_1, T_2)\}\mid T_1, T_2\right] < C_2,
\end{equation*}
almost surely.
\end{lemma}

\begin{proof}
Consider the increments
\begin{align*}
    Y_1 = \frac{X_1-m(0,0,T_1)}{\sqrt{\Phi^2(0,T_1)v_0 + v(0, T_1)}}, && Y_2 = \frac{X_2 - m(X_1, T_1, T_2)}{\sqrt{v(T_1, T_2)}}, 
\end{align*}
as well as
\begin{align*}
    Z_1 = \frac{X_1-m_0(0,0,T_1)}{\sqrt{\Phi_0^2(0,T_1)v_0 + v_0(0, T_1)}}, && Z_2 = \frac{X_2 - m_0(X_1, T_1, T_2)}{\sqrt{v_0(T_1, T_2)}}.
\end{align*}
Note that conditionally on $(T_1, T_2)$, the variables $Z_1$ and $Z_2$ are independent standard normal under the true probability measure. We observe that the log density satisfies
\begin{align*}
    \abs*{\ell_{\theta}-\ell_{\underline{\theta}}} &\leq \frac{1}{2}\abs{\log \frac{\underline{v}}{v}} + \frac{1}{2}\abs*{\frac{(X_2 - m)^2}{v} - \frac{(X_2 - \underline{m})^2}{\underline{v}}} \\
    & \lesssim \frac{\abs{v-\underline{v}}}{v} + \abs*{(X_2 - m)^2 \frac{\underline{v}-v}{v\underline{v}} + \frac{(X_2-m)^2 - (X_2-\underline{m})^2}{\underline{v}}}\\
    &\lesssim  \frac{\abs{v-\underline{v}}}{v} + \frac{(X_2 - m)^2}{v} \,\frac{\abs{v - \underline{v}}}{\underline{v}} + \frac{\abs{(m-\underline{m})(2(X_2 - m) + m - \underline{m})}}{\underline{v}}\\
    &\lesssim  \frac{\abs{v-\underline{v}}}{v} + \frac{(X_2 - m)^2}{v} \,\frac{\abs{v - \underline{v}}}{\underline{v}} + \frac{\abs{X_2-m}}{\sqrt{\underline{v}}}\frac{\abs{m-\underline{m}}}{\sqrt{\underline{v}}} + \frac{(m-\underline{m})^2}{\underline{v}}\\
    &\lesssim  (1+Y_2^2)\frac{\abs{v-\underline{v}}}{T_2 - T_1} + \abs{Y_2}\frac{\abs{m-\underline{m}}}{\sqrt{T_2-T_1}} + \frac{(m-\underline{m})^2}{T_2-T_1},
\end{align*}
uniformly in the variables $(X_1, X_2, T_1,T_2)$. Furthermore, it can be easily seen that
\begin{equation*}
    \frac{\abs{m - \underline{m}}}{\sqrt{T_2-T_1}} \lesssim (1 + \abs{Y_1}) \frac{\abs{\Phi(T_1, T_2) - \underline{\Phi}(T_1, T_2)}}{\sqrt{T_2 - T_1}} +\frac{\abs{m(0, T_1, T_2) - \underline{m}(0, T_1, T_2)}}{\sqrt{T_2 - T_1}}.
\end{equation*}
Combining the above, it can be shown that
\begin{equation*}
    \abs*{\ell_{\theta}-\ell_{\underline{\theta}}} \lesssim  \left(1 + Y_1^2 + Y_2^2\right) G(\theta, \underline{\theta};T_1, T_2) \lesssim \left(1 + Z_1^2 + Z_2^2\right) G(\theta, \underline{\theta};T_1, T_2).
\end{equation*}
Consider $F(X_1, X_2, T_1, T_2) = 1 + Z_1^2 + Z_2^2$ and let $C_1 > 0$ be arbitrary. Because the conditional distributions of $Z_1^2$ and $Z_2^2$ have sub-exponential tails, there exists some $\kappa > 0$ small enough and a constant $C_2 > 0$ which is sufficiently large, such that
\begin{align*}
  &\expv\left[ \exp\left\lbrace\kappa C_1 F(X_1, X_2, T_1, T_2)\right\rbrace \mid T_1, T_2\right] \\
  &\qquad = \exp\left\lbrace \kappa C_1 \right\rbrace\expv\left[ \exp\left\lbrace\kappa C_1 Z_1^2\right\rbrace\expv\left[\exp\left\lbrace\kappa C_1Z_2^2\right\rbrace\mid T_1, T_2, Z_1 \right]  \mid T_1,T_2\right] \\
  &\qquad = \exp\left\lbrace \kappa C_1 \right\rbrace \expv\left[\exp\left\lbrace\kappa C_1 Z_1^2\right\rbrace \mid T_1\right]\expv\left[\exp\left\lbrace\kappa C_1Z_2^2\right\rbrace\mid T_1, T_2\right]  < C_2,
\end{align*}
almost surely. We used the fact that $Z_2$ is conditionally independent of $Z_1$ given $T_1, T_2$.
\end{proof}

\begin{lemma}\label{lem:G}
Consider $G$ from Lemma~\ref{lem:FG}. Then,
\begin{enumerate}[label=(\roman*)]
 \item $G(\theta, \underline{\theta};T_1, T_2) \lesssim \norm{b_1 - \underline{b}_1}_\infty + \norm{b_2 - \underline{b}_2}_\infty + \norm{\sigma^2 - \underline{\sigma}^2}_\infty$,
 \item $\expv\left[G(\theta, \underline{\theta};T_1, T_2)^2\right] \lesssim \normw{b_1 - \underline{b}_1}^2 + \normww{b_2 - \underline{b}_2}^2 + \normwww{\sigma^2 - \underline{\sigma}^2}^2$.
\end{enumerate}
\end{lemma}

\begin{proof}
    The proof of (i) is straightforward. Concerning (ii), note that the terms in Lemma~\ref{lem:Rs} yield
    \begin{equation*}
        \expv\left[G(\theta, \underline{\theta};T_1, T_2)^2\right] \lesssim R_1 + R_2 + R_3 + R_4 + R_5,
    \end{equation*}
    and the assertion then follows from the corresponding upper bounds.
\end{proof}

\section{Comparison of norms}\label{sec:Norms}

\begin{lemma}\label{lem:UpperBoundNorm}
    For $f \in L^2[0,1]$, it holds that
    \begin{equation*}
        \normww{f} \leq \normwww{f} \leq \sqrt{\log(2)} \, \norm{f}_{L^2}.
    \end{equation*}
\end{lemma}

\begin{proof}
The inequality $ \normww{f} \leq \normwww{f}$ follows from a simple comparison of the weights. For the second inequality, the Cauchy-Schwarz inequality implies
    \begin{align*}
    \normwww{f}^2 &= \int_0^1 \int_0^{t_2} \frac{1}{(t_2 - t_1)^2} \left(\int_{t_1}^{t_2} f(u) \, \diff u\right)^2 \,\diff t_1\, \diff t_2 \\
    &\leq \int_0^1 \int_0^{t_2} \frac{1}{t_2 - t_1} \int_{t_1}^{t_2} f^2(u) \,\diff u \, \diff t_1 \, \diff t_2.
\end{align*}
Changing the order of integration and using the inequality $\int_u^{1} \int_0^u  \frac{1}{t_2-t_1}  \, \diff t_1 \, \diff t_2 \leq \log(2)$,
\begin{equation*}
    \normwww{f}^2 \leq \int_0^1 f^2(u) \int_u^{1} \int_0^u  \frac{1}{t_2-t_1}  \, \diff t_1 \, \diff t_2 \,\diff  u\leq \int_0^1 f^2(u) \log(2) \, \diff u \leq \log(2) \norm{f}_{L^2}^2,
\end{equation*}
which completes the proof.
\end{proof}

\begin{lemma}\label{lem:LowerBoundNorm}
    If $f\in L^2[0,1]$ satisfies the finite basis expansion $f(u) = \sum_{k=1}^M f_k \varphi_k(u)$, there exists some $c_6 > 0$, not depending on $M$, such that
\begin{align*}
    \normww{f}^2 \geq \frac{c_6}{M^{2}} \norm{f}_{L^2}^2, && 
    \normwww{f}^2 \geq \frac{c_6}{M^1} \norm{f}_{L^2}^2.
\end{align*}
\end{lemma}

\begin{proof}
For $\alpha, \gamma \geq 0$ such that $3 + \alpha - \gamma > 0$, consider the norm
\begin{equation*}
    \norm{f}_{\alpha, \gamma} = \sqrt{\int_0^1 \int_0^{t_2} \frac{t_1^\alpha}{(t_2-t_1)^\gamma}\left(\int_{t_1}^{t_2} f(u) \, \diff u\right)^2 \,\diff t_1\, \diff t_2}.
\end{equation*}
The assertion of the lemma then follows from the more general result
\begin{equation*}
    \norm{f}_{\alpha, \gamma}^2 \geq \frac{c_6}{M^{3+\alpha-\gamma}} \norm{f}_{L^2}^2.
\end{equation*}
Toward this end, define $(\mathcal{A}_h f)(t) = \frac{1}{h} \int_0^{h} f(t + u) \, \diff u$, for $t \in [h,1-h]$, and observe
\begin{equation}\label{eq:NormAlpha}
    \norm{f}_{\alpha, \gamma}^2 \geq \int_0^1 h^{2+\alpha-\gamma} \norm{\mathcal{A}_h f}^2_{L^2(h,1-h)} \, \diff h.
\end{equation}
An argument based on the Minkowski inequality shows
\begin{equation*}
    \norm{\mathcal{A}_h f - f}_{L^2(h, 1-h)} \leq \frac{h}{2} \norm{f'}_{L^2}.
\end{equation*}
Since $\norm{f'}_{L^2} \leq \pi M \norm{f}_{L^2}$, it follows that
\begin{equation}\label{eq:Ah}
    \norm{\mathcal{A}_h f - f}_{L^2(h, 1-h)} \leq \frac{\pi M h}{2} \norm{f}_{L^2}.
\end{equation}
In addition, the Cauchy--Schwarz inequality yields $\norm{f}_\infty^2 \leq 2 M \norm{f}_{L^2}^2$ which implies
\begin{equation}\label{eq:fh}
    \norm{f}_{L^2(h, 1-h)}^2 \geq \norm{f}_{L^2}^2 - 2 h \norm{f}_\infty^2 \geq (1-4Mh) \norm{f}_{L^2}^2.
\end{equation}
Combining \eqref{eq:Ah} and \eqref{eq:fh}, one obtains
\begin{equation*}
    \norm{\mathcal{A}_h f}_{L^2(h,1-h)}  \geq \left(\sqrt{1-4Mh} - \frac{\pi Mh}{2}\right)\norm{f}_{L^2} \geq \frac{1}{2} \norm{f}_{L^2},
\end{equation*}
provided that $h \leq 1/(8M)$. Plugging the above into~\eqref{eq:NormAlpha},
\begin{equation*}
    \norm{f}_{\alpha,\gamma}^2 \geq \int_0^{1/(8M)} h^{2+\alpha-\gamma} \norm{\mathcal{A}_h f}^2_{L^2(h,1-h)} \, \diff h \geq \int_0^{1/(8M)} h^{2+\alpha-\gamma} \frac{1}{4} \norm{f}_{L^2}^2 \, \diff h \geq \frac{c_6}{M^{3+\alpha-\gamma}} \norm{f}_{L^2}^2, 
\end{equation*}
for some constant $c_6 > 0$ which does not depend on $M$.
\end{proof}

\begin{lemma}\label{lem:LowerBoundNormII}
Let $\sigma(u) = \exp(\sum_{k=1}^M \sigma_k \varphi_k(u))$ and $\underline{\sigma}(u) = \exp(\sum_{k=1}^M \underline{\sigma}_k \varphi_k(u))$. There exists some $c_7 > 0$, not depending on $M$, such that
\begin{equation*}
    \normwww{\sigma^2-\underline{\sigma}^2}^2 \geq \frac{c_7}{M} \norm{\sigma^2-\underline{\sigma}^2}_{L^2}^2.
\end{equation*}
\end{lemma}

\begin{proof}
Set $f(u) = \sigma^2(u) - \underline{\sigma}^2(u)$ for $u \in [0,1]$. Then, $\norm{f'}_{L^2} \lesssim M \norm{f}_{L^2}$ and $\norm{f}_{\infty}^2 \lesssim M \norm{f}_{L^2}^2$. The assertion then follows from arguments similar to those used in the proof of Lemma~\ref{lem:LowerBoundNorm} for the case $\alpha = 0$ and $\gamma = 2$.
\end{proof}

\putbib[ref]
\end{bibunit}

\end{document}